\documentclass[11pt,reqno]{amsart}
\usepackage{amssymb,amsmath,amsthm,amsrefs,enumerate,verbatim,bbm}
\usepackage[a4paper]{geometry}
\usepackage[utf8]{inputenc}
\usepackage[english]{babel}
\usepackage{amsfonts}
\usepackage{dsfont}
\usepackage[pdftex]{graphicx}
\usepackage{mathrsfs}
\usepackage[mathscr]{euscript}
\usepackage{microtype}
\usepackage{amssymb}
\usepackage{amsthm}
\usepackage{amsmath}
\usepackage{a4wide}
\usepackage{latexsym}

\usepackage[v2,tips]{xy}
\usepackage{float}
\usepackage[T1]{fontenc}
\usepackage{lmodern}
\usepackage{xcolor}

\makeatletter
\@namedef{subjclassname@2010}{%
  \textup{2010} Mathematics Subject Classification}
\makeatother

\title[Algebras of operators on long sequence spaces]{Surjective homomorphisms from algebras\\ of operators on long sequence spaces\\ are automatically injective}

\newtheorem{Thm}{Theorem}[section]
\newtheorem*{tha*}{{\textbf{Theorem A}}}
\newtheorem*{thmb*}{{\textbf{Theorem B}}}
\newtheorem{Lem}[Thm]{Lemma}
\newtheorem{Prop}[Thm]{Proposition}
\newtheorem{Cor}[Thm]{Corollary}

\theoremstyle{definition}

\newtheorem{Rem}[Thm]{Remark}
\newtheorem{Que}[Thm]{Question}

\newtheorem*{Ack}{Acknowledgement}
\newtheorem*{Fun}{Funding}

\numberwithin{equation}{section}

\DeclareMathOperator*{\Ran}{\mathrm{im}}
\DeclareMathOperator*{\Ker}{\mathrm{ker}}

\DeclareMathOperator*{\id}{\mathrm{id}}

\DeclareMathOperator*{\spanning}{\mathrm{span}}

\DeclareMathOperator*{\support}{\mathrm{supp}}
\DeclareMathOperator*{\cf}{\mathrm{cf}}

\newcommand{\vertiii}[1]{{\left\vert\kern-0.25ex\left\vert\kern-0.25ex\left\vert #1 
    \right\vert\kern-0.25ex\right\vert\kern-0.25ex\right\vert}}

\newcounter{smallromans}

\newenvironment{romanenumerate}
{\begin{list}{{\normalfont\textrm{(\roman{smallromans})}}}%
  {\usecounter{smallromans}\setlength{\itemindent}{0cm}%
   \setlength{\leftmargin}{5.5ex}\setlength{\labelwidth}{5.5ex}%
   \setlength{\topsep}{.5ex}\setlength{\partopsep}{.5ex}%
   \setlength{\itemsep}{0.1ex}}}%
{\end{list}}

\newcounter{smallromansdash}

{\end{list}}

\newcounter{bigromans} 
  {\end{list}}

\date{\today}

\author{Bence Horváth}
\address{Institute of Mathematics, Czech Academy of Sciences, \v{Z}itn\'{a} 25, 115~67 Prague 1, Czech Republic}
\email{horvath@math.cas.cz, hotvath@gmail.com}

\author{Tomasz Kania}
\address{Institute of Mathematics, Czech Academy of Sciences, \v{Z}itn\'{a} 25, 115~67 Prague 1, Czech Republic and Institute of Mathematics, Jagiellonian University, {\L}ojasiewicza 6, 30-348 Krak\'{o}w, Poland}

\email{kania@math.cas.cz, tomasz.marcin.kania@gmail.com}



\begin{document}

\maketitle

\begin{abstract}
We study automatic injectivity of surjective algebra homomorphisms from $\mathscr{B}(X)$, the algebra of (bounded, linear) operators on $X$, to $\mathscr{B}(Y)$, where $X$ is one of the following \emph{long} sequence spaces: $c_0(\lambda)$, $\ell_{\infty}^c(\lambda)$, and $\ell_p(\lambda)$ ($1 \leqslant p < \infty$) and $Y$ is arbitrary. \textit{En route} to the proof that these spaces do indeed enjoy such a property, we classify two-sided ideals of the algebra of operators of any of the aforementioned Banach spaces that are closed with respect to the `sequential strong operator topology'.
\end{abstract}

\section{Introduction and known results}	

Algebras of operators on Banach spaces are quite rigid as illustrated by Eidelheit's Theorem {(see \cite[Theorem~2.5.7]{Dales})}, which asserts that for two Banach spaces $X$ and $Y$, the algebras of operators $\mathscr{B}(X)$ and $\mathscr{B}(Y)$ on the respective spaces are isomorphic as rings/Banach algebras precisely when $X$ and $Y$ are isomorphic as Banach spaces. Thus, in a sense, lots of isomorphic Banach space theory may be translated to algebraic problems concerning the algebras $\mathscr{B}(X)$ and \emph{vice versa}. \smallskip 

As observed by the first-named author in \cite{horvath2}, for many Banach spaces $X$ such as $c_0$ or $\ell_p$ ($p\in [1,\infty])$, a~stronger rigidity phenomenon is available: for every non-zero Banach space $Y$, every surjective algebra homomorphism $\varphi\colon \mathscr{B}(X)\to \mathscr{B}(Y)$ is automatically injective, that is, it is an algebra isomorphism. In the said paper, such spaces have been termed to have the SHAI property after \emph{surjective homomorphisms are injective}, and we continue using this terminology. Interestingly, not every Banach space enjoys such a property; spaces $X$ whose algebra $\mathscr{B}(X)$ admits a character (a non-zero homomorphism into the scalar field) are obvious counter-examples (historically, the first two examples of such spaces are $X = J_2$, the James space (\cite[paragraph~8]{edmit}; see also \cite[Theorem~4.16]{laustsenmax1}), and $X=C[0,\omega_1]$, the space of continuous functions on the ordinal interval $[0,\omega_1]$). \smallskip

Let us list positive results concerning SHAI from \cite{horvath2}. The following Banach spaces have the SHAI property:
\begin{romanenumerate}
    \item \label{ellpshi} $c_0$ and $\ell_p$ for $p\in [1,\infty]$;
    \item \label{hilbshi} Hilbert spaces of arbitrary density, \emph{e.g.}, $\ell_2(\Gamma)$ for any set $\Gamma$;
    \item \label{arbdistortshi} complementably minimal spaces $X$ that contain a complemented subspace isomorphic to $X\oplus X$ (in particular, Schlumprecht's arbitrarily distortable Banach space $S$);
    \item $X = (\bigoplus_{n=1}^\infty \ell_2^n)_{c_0}$ and $X = (\bigoplus_{n=1}^\infty \ell_2^n)_{\ell_1}$.
\end{romanenumerate}
Moreover, if both Banach spaces $X_1$ and $X_2$ have the  SHAI property, then so has $X=X_1\oplus X_2$.\smallskip

In particular, it follows that direct sums such as $c_0 \oplus \ell_p$ and $\ell_p \oplus \ell_q$ have the SHAI property for every $1 \leqslant p,q \leqslant \infty$. The importance of this result is that $\mathscr{B}(\ell_p \oplus \ell_q)$ has a very complicated ideal structure (see \cite{fschzs2, fschzs, schzs}) and the study of automatic injectivity of surjective homomorphisms is intimately related with their kernels that are closed ideals of $\mathscr{B}(X)$ themselves.\smallskip

On the negative side, let us record the following results here for the sake of completeness, established by the first-named author in \cite{horvath2}. These are \cite[Lemma~2.2]{horvath2} and \cite[Theorem~1.7]{horvath2}, respectively. Note that the first result we already invoked in the case where $\dim Y =1$.

\begin{romanenumerate}
    \item Let $X$ be an infinite-dimensional Banach space such that $\mathscr{B}(X)$ admits a~finite-dimensional quotient. Then $X$ does not have the SHAI property. In particular, the James space, $C[0,\omega_1]$, and hereditarily indecomposable Banach spaces and finite direct sums thereof fail to have SHAI.
    
    \item Let $X$ be a non-zero, separable, reflexive Banach space, and consider the injective tensor product $Y_X:= C_0[0, \omega_1) \hat{\otimes}_{\varepsilon} X$. There exist a non-injective contractive algebra homomorphism $\Theta\colon  \mathscr{B}(Y_X) \rightarrow \mathscr{B}(X)$ and a contractive algebra homomorphism $\Lambda\colon  \mathscr{B}(X) \rightarrow \mathscr{B}(Y_X)$ such that $\Theta \circ \Lambda = \id_{\mathscr{B}(X)}$. In particular, $\Theta$ is surjective.
\end{romanenumerate}

In \cite{horvath2}, a promise concerning establishing the SHAI property for Banach spaces of the form $\ell_p(\Gamma)$ for any $p\in [1,\infty)$ and every set $\Gamma$ was made. The aim of this paper is to fulfil this promise actually for a larger class of Banach spaces that we collectively call \emph{long}.

\begin{tha*}\label{longseqshai}
	The Banach spaces $c_0(\lambda)$, $\ell_{\infty}^c(\lambda)$, and $\ell_p(\lambda)$ have the SHAI property for every infinite cardinal $\lambda$ and every $p \in [1, \infty)$.
\end{tha*}

Along the way, we establish new results concerning the lattice of closed ideals of $\mathscr{B}(\ell_p(\Gamma))$ for any set $\Gamma$, introduce and use a certain topology that we term the $\sigma$-strong operator topology (denoted by $\sigma_{\rm SOT}$), which for long sequence spaces (that is, when $\Gamma$ is uncountable), is intermediate between the strong operator topology and the norm topology.\smallskip

Among other things, we prove the following result concerning the set $\mathscr{S}_{E_\kappa}(X)$ of operators in $\mathscr{B}(X)$ that do not preserve isomorphic copies of $E_\kappa$, where $E_\kappa$ is one of the long sequence spaces considered in the present paper.

\begin{thmb*}\label{uncountcof}
	Let $X$ be a Banach space and let $\kappa$ be a cardinal number with uncountable cofinality. Consider one of the following cases:
	\begin{enumerate}
		\item $E_{\kappa}:= c_0(\kappa)$ and $X$ has an M-basis;
		\item $E_{\kappa}:= \ell_p(\kappa)$ and $X:= \ell_p(\lambda)$, where $\lambda \geqslant \kappa$ and $p \in (1,\infty)$;
		\item $E_{\kappa}:= \ell_1(\kappa)$.
	\end{enumerate}
		Then the set $\mathscr{S}_{E_{\kappa}}(X)$ is $\sigma_{\rm SOT}$-closed in $\mathscr{B}(X)$.
\end{thmb*}

Furthermore, with the aid of a striking new result of Koszmider and Laustsen from \cite{lkad}, we prove in Proposition \ref{not3sp} results related to the three-space problem (for example, that SHAI is not a three space property of Banach spaces, even though it is preserved by finite direct sums). Some questions related to the SHAI property of certain Banach spaces are also left open.\pagebreak

\section{Preliminaries} 
For a set $S$, we denote by $\mathscr{P}(S)$ the power-set of $S$. The symbol $[S]^2$ stands for the subset of $\mathscr{P}(S)$ whose elements consist of exactly two elements of $S$.\smallskip

For a function $f\colon X\to Y$, we denote by ${\rm im}\, f$ the image of $f$. For a set $\widehat{Y}\supseteq {\rm im}\, f$, we denote by $f|^{\widehat{Y}}$ the corestriction of $f$ to $\widehat{Y}$, that is, we consider it a map $f\colon X\to \widehat{Y}$. For a subset $\widehat{X}\subseteq X$, we denote by $f|_{\widehat{X}}$ the restriction of $f$ to $\widehat{X}$.\smallskip

We use von Neumann's approach to ordinal and cardinal numbers; for example, we consider the latter initial ordinal numbers. If $\kappa$ is a cardinal number then $\kappa^+$ denotes its successor cardinal. \emph{Cofinality} of a set of ordinal numbers $\Lambda$, $\cf(\Lambda)$, is the least cardinality of a cofinal subset of $\Lambda$. A cardinal number is \emph{regular}, whenever it is equal to its cofinality. The following lemma is standard, see, \emph{e.g.}, \cite[Lemma~3.2]{dawsideal}. 

\begin{Lem}\label{cofinal}
	Let $\kappa$ be a cardinal number with $\cf(\kappa) > \omega$. Let $(\Lambda_n)_{n=1}^\infty$ be a sequence of sets such that $| \Lambda_n | < \kappa$ for all $n \in \mathbb{N}$. Then $| \bigcup_{n=1}^\infty  \Lambda_n | < \kappa$.
\end{Lem}

Let $\mathbb K$ denote the field of real or complex numbers. Let $\Gamma$ be a set and $p\in [1,\infty]$. When $p<\infty$, we denote by $\ell_p(\Gamma)$ the space of all functions $f\colon \Gamma \to \mathbb K$ with $\sum_{\gamma\in \Gamma} |f(\gamma)|^p<\infty$ normed by the $1/p$\textsuperscript{th} power of this expression. When $p=\infty$, $\ell_\infty(\Gamma)$ stands for the space of all bounded functions $f\colon \Gamma \to \mathbb K$ normed by the supremum norm. When $\Gamma$ is uncountable, $\ell_\infty^c(\Gamma)$ stands for the (closed) subspace of $\ell_\infty(\Gamma)$ comprising functions for which the set ${\rm supp}\, f = \{\gamma\in \Gamma\colon f(\gamma)\neq 0\}$ is finite or countably infinite. The symbol $c_0(\Gamma)$ denotes the space of all functions $f\colon \Gamma \to \mathbb K$ such that the set $\{\gamma \in \Gamma \colon |f(\gamma)| \geqslant \varepsilon\}$ is finite for every $\varepsilon >0$. It is a standard fact that all the aforementioned spaces are complete. Whenever $\Gamma$ is uncountable, we collectively call the spaces $c_0(\Gamma), \ell_p(\Gamma)$ and $\ell_\infty^c(\Gamma)$ \emph{long sequence spaces}.\medskip

\subsubsection{Operator ideals}\label{sect: op ids}
Let $X$ and $Y$ be Banach spaces. We denote by $\mathscr{B}(X,Y)$ the space of all (bounded, linear) operators from $X$ to $Y$, which is a Banach space under the operator norm. In particular, $\mathscr{B}(X):=\mathscr{B}(X,X)$ is a Banach algebra. We shall be primarily interested in surjective algebra homomorphisms $\varphi\colon \mathscr{B}(X)\to \mathscr{B}(Y)$, which are known to be automatically continuous due to the fundamental result of B.~E.~Johnson, see for example \cite[Theorem~5.1.5]{Dales}.\smallskip

For a Banach space $X$, $\mathscr{F}(X), \mathscr{A}(X), \mathscr{K}(X), \mathscr{W}(X)$ stand for the ideals of $\mathscr{B}(X)$ comprising finite-rank operators, approximable operators (operators in the closure of $\mathscr{F}(X)$), compact operators, and weakly compact operators, respectively. We denote by $\mathscr{X}(X)$ the ideal of operators that have separable range and by $\mathscr{E}(X)$ the ideal of inessential operators, that is, operators $T\in \mathscr{B}(X)$ such that for any $A\in \mathscr{B}(X)$ both operators $I_X - AT, I_X - TA$ are Fredholm.\smallskip

For fixed Banach spaces $X,Y$, and $Z$ the symbol $\mathscr{S}_Z(Y,X)$ denotes the subset of those operators in $\mathscr{B}(Y,X)$ which are not bounded below on any subspace of $Y$ isomorphic to $Z$. In other words, for $T \in \mathscr{B}(Y,X)$ we have $T \notin \mathscr{S}_Z(Y,X)$ if and only if there is a closed subspace $W$ of $Y$ such that $W \cong Z$ and $T |_W$ is bounded below, that is, there exists $\gamma > 0$ such that $\|Tw\| \geqslant \gamma \|w\|$ for all $w\in W$. We also make use of the abbreviation $\mathscr{S}_Y(X) := \mathscr{S}_Y(X,X)$. The elements of the set $\mathscr{S}(Y,X)$, defined as the intersection of all $\mathscr{S}_Z(Y,X)$, where $Z$ ranges through all infinite-dimensional subspaces of $Y$, are called \emph{strictly singular} operators. It is well-known that for every Banach space $X$ one has the inclusions
$$\mathscr{A}(X)\subseteq \mathscr{K}(X)\subseteq \mathscr{S}(X)\subseteq \mathscr{E}(X)\qquad \text{ and }\qquad \mathscr{A}(X)\subseteq \mathscr{K}(X)\subseteq \mathscr{W}(X).$$

For Banach spaces $X$ and $Y$, the set $\mathscr{S}_Y(X)$ is closed under multiplication from the left and from the right by arbitrary operators in $\mathscr{B}(X)$. However, $\mathscr{S}_Y(X)$ need not be closed under addition. To see this, let us consider, for example, $X = \ell_p \oplus \ell_q$, where $1 \leqslant q < p < \infty$, in which case the projection on the respective summands are in $\mathscr{S}_X(X)$ but their sum is not. It is also obvious that $\mathscr{S}(X) \subseteq \mathscr{S}_Y(X) \subseteq \mathscr{S}_Z(X)$ for infinite-dimensional Banach spaces $X,Y$ and $Z$ with $Y \subseteq Z$.

Lastly, if $X$ and $Y$ are Banach spaces with $Y$ non-separable, then $\mathscr{X}(X) \subseteq \mathscr{S}_Y(X)$. Indeed, if $T \in \mathscr{B}(X)$ is such that $T \notin \mathscr{S}_Y(X)$ then there is a closed subspace $Z$ of $X$ with $Z \cong Y$ such that $T |_Z$ is bounded below. Hence
\[
Y \cong Z \cong \Ran(T |_Z) \subseteq \Ran(T),
\]
which shows that $Y$ embeds into $\Ran(T)$, so that $\Ran(T)$ cannot be separable.\medskip

\subsubsection{Complementably homogenous Banach spaces} An infinite-dimensional Banach space is\index{complementably homogenous} \textit{complementably homogenous}, whenever for every closed subspace $Y$ of $X$ with $Y \cong X$ there exists a complemented subspace $W$ of $X$ with $W \cong X$ and $W \subseteq Y$. The spaces $c_0$ and $\ell_p$ (where $1 \leqslant p < \infty$) are well known to be complementably homogenous; this follows, for example, from \cite[Lemma~2]{pelczynski1}. When $\lambda$ is an uncountable cardinal, $c_0(\lambda)$ and $\ell_p(\lambda)$ are also known to be complementably homogenous. These results follow for example, from \cite[Proposition~2.8]{acgjm} and \cite[Proposition~3.10]{jksch}, respectively. For every infinite cardinal number $\lambda$, the Banach space $\ell_{\infty}^c(\lambda)$ is complementably homogenous too (see \cite[Theorem~1.2]{jksch}). \smallskip

We would like to draw the reader's attention to the paper \cite{rodsal} of Rodr\'{\i}guez-Salinas, which seems to be a bit overlooked. The author had already shown in this paper that for an infinite cardinal number $\lambda$, every complemented subspace of $\ell_p(\lambda)$ (for $1<p< \infty$) is isomorphic to $\ell_p(\kappa)$ for some cardinal $\kappa \leqslant \lambda$ (see \cite[Theorem~4]{rodsal}).\smallskip

The following lemma is a slight generalisation of Whitley's result \cite[Theorem~6.2]{whitley}. 

\begin{Lem}\label{complhomideal}
	Suppose that $X$ is a complementably homogenous Banach space. Let $\mathscr{J}$ be a~subset of $\mathscr{B}(X)$ that is closed under multiplication from the left and from the right by arbitrary operators in $\mathscr{B}(X)$. If $\mathscr{J}$ is a proper subset of $\mathscr{B}(X)$, then $\mathscr{J} \subseteq \mathscr{S}_X(X)$.
\end{Lem}

\begin{proof} 
	Assume that $\mathscr{J} \nsubseteq \mathscr{S}_X(X)$ and take $T \in \mathscr{J}$ such that $T \notin \mathscr{S}_X(X)$. Then there exists a~subspace $W$ of $X$ such that $W \cong X$ and $T \vert_W$ is bounded below. Set $T_1 := T \vert_{W}^{\Ran(T \vert_W)}$, then $T_1 \in \mathscr{B}(W, \Ran(T \vert_W))$ is an isomorphism. In particular, $\Ran(T \vert_W) \cong W \cong X$. Since $X$ is complementably homogenous, there exists an idempotent $P \in \mathscr{B}(X)$ with $\Ran(P) \cong X$ and $\Ran(P) \subseteq \Ran(T \vert_W)$. Let $S \in \mathscr{B}(\Ran(P),X)$ be an isomorphism and let $\iota \colon \, W \rightarrow X$ denote the canonical embedding. Since $\Ran(P) \subseteq \Ran(T \vert_W)$, clearly $T_1^{-1} \vert_{\Ran(P)} \in \mathscr{B}(\Ran(P),W)$. It is therefore immediate that
	\begin{align}
	(S \circ P \vert^{\Ran(P)}) \circ T \circ (\iota \circ T_1^{-1} \vert_{\Ran(P)} \circ S^{-1}) = S \circ P \vert^{\Ran(P)} \circ P \vert_{\Ran(P)} \circ S^{-1} =I_X.
	\end{align}	
	Consequently, as $T \in \mathscr{J}$, it follows that $I_X \in \mathscr{J}$, equivalently, $\mathscr{J} = \mathscr{B}(X)$. 
\end{proof}

\begin{Cor}\label{complhommaxideal}
	Let $X$ be a complementably homogenous Banach space. Then $\mathscr{S}_X(X)$ is the unique maximal two-sided ideal of $\mathscr{B}(X)$ if and only if $\mathscr{S}_X(X)$ is closed under addition.
\end{Cor}

Using Lemma~\ref{complhomideal}, it is possible to give an alternative proof of the fact that the algebras of bounded operators on $c_0$ and $\ell_p$ ($p\in [1,\infty)$) have only one non-trivial closed two-sided ideal, namely $\mathscr{S}_X(X)$. Even though the result is well-known, its proof is hard to find in the literature, so we take Lemma~\ref{complhomideal} as an excuse for presenting the proof here in full detail.

\begin{Cor}\label{classicalideal}
	Let $X:=c_0$ or $X:= \ell_p$, where $1 \leqslant p < \infty$. Then $\mathscr{A}(X) = \mathscr{E}(X) = \mathscr{S}_X(X)$. If $X:= \ell_{\infty}$, then $\mathscr{E}(X) = \mathscr{X}(X) = \mathscr{S}_X(X)$. In either case, $\mathscr{S}_X(X)$ is the unique maximal two-sided ideal of $\mathscr{B}(X)$.
\end{Cor}

\begin{proof}
	We have $\mathscr{A}(X) \subseteq \mathscr{E}(X)$. As $X$ is complementably homogenous, by Lemma \ref{complhomideal}, $\mathscr{E}(X) \subseteq \mathscr{S}_X(X)$.
	Suppose first that $X= c_0$ or $X= \ell_p$, where $1 \leqslant p < \infty$. Let $T \in \mathscr{B}(X)$ be such that $T \notin \mathscr{A}(X)$. By \cite[Section~5.1.1, Lemma~3]{pietsch}, there exist $R,S \in \mathscr{B}(X)$ with $I_X = RTS$. As $R$ and $T$ are non-zero, it is immediate that $S$ is bounded below on $X$ and thus $Z:=\Ran(S) \cong X$. We observe that $T |_Z$ is bounded below. Indeed, let $z \in Z$ be arbitrary and pick $x \in X$ with $z=Sx$. Then, indeed,
	\begin{align}
	\|z\| = \|Sx\| \leqslant \|S\| \|x\| = \|S\| \| RTS x\| = \|S\| \| RTz\| \leqslant \|S\| \|R\| \|Tz\|.
	\end{align}
	This together with $Z \cong X$ yields $T \notin \mathscr{S}_X(X)$. Thus $\mathscr{S}_X(X) \subseteq \mathscr{A}(X)$.
	When $X= \ell_{\infty}$, this is explained in detail in \cite[page~253]{laustsenloy}.
\end{proof}

\section{Auxiliary results and the proofs of Theorem A \& B}

Let $X$ be a Banach space. An indexed collection $(x_i, f_i)_{i\in J}$ in $X\times X^*$ is a \emph{biorthogonal system}, whenever $\langle x_i, f_j \rangle = \delta_{i,j}$ for $i,j\in J$. A biorthogonal system $(x_i, f_i)_{i\in J}$ is an M-basis, whenever $\{x_i\colon i\in J\}$ is \emph{fundamental} (linearly dense in $X$) and $\{f_i\colon i\in J\}$ is \emph{total} (linearly weak*-dense in $X^*$). For a collection $\Phi := (f_j)_{j\in J}$ in $X^*$, the \textit{support} of $x \in X$ with respect to $\Phi$ is defined as
\[
{\support}_\Phi(x):= \{j\in J\colon \langle x, f_j \rangle \neq 0 \},
\]
however we usually drop the subscript $\Phi$ when the considered collection is clear from the context (for example, when there is a fixed M-basis for $X$). \smallskip 

In $c_0(\Gamma)$ and $\ell_p(\Gamma)$ for $p\in [1,\infty]$, by default, we consider the supports with respect to the evaluation functionals at points $\gamma\in \Gamma$; the notion of support defined in this way agrees with the definition of the support introduced earlier. The functionals themselves are coordinate functionals corresponding to the standard unit vector basis $(e_{\alpha})_{\alpha \in \Gamma}$ of either space, apart from $\ell_\infty(\Gamma)$.
When $(x_i, f_i)_{i\in J}$ is an M-basis for $X$, the collection $\Phi:=(f_j)_{j\in J}$ is \emph{countably supporting}, that is, the set ${\support}_{\Phi}(x)$ is countable for each $x\in X$. 

\begin{Prop}\label{disjointimage}
	Suppose that $X$ is a Banach space and $\kappa$ is an uncountable cardinal number. Let $(x_{\alpha})_{\alpha < \kappa}$ be a transfinite basic sequence in $X$ equivalent to the standard unit vector basis of $c_0(\kappa)$ or $\ell_p(\kappa)$, where $p \in (1, \infty)$. Let $Y$ be a Banach space that has an M-basis. If $T \in \mathscr{B}(X,Y)$ is non-zero, then there exists $\Lambda \subseteq \kappa$ with $|\Lambda | = \kappa$ such that $(Tx_{\alpha})_{\alpha \in \Lambda}$ consists of disjointly supported vectors.
\end{Prop}

\begin{proof}
	
	By \cite[Theorem~5.13]{hajek}, without loss of generality we may assume that $(b_j, f_j)_{j \in J}$ is an~M-basis for $Y$ with $\sup_{j \in J} \| f_j \| \leqslant K$ for some $K >0$.\smallskip
	
	For $\alpha < \kappa$, set $y_{\alpha} := Tx_{\alpha}$. By the Kuratowski--Zorn Lemma we can take a set $\Lambda \subseteq \lambda$ which is maximal with respect to the property that the vectors $y_{\alpha}$ and $y_{\beta}$ are disjointly supported for each distinct $\alpha, \beta \in \Lambda$. Assume towards a~contradiction that $|\Lambda| < \kappa$. Let $\Gamma := \textstyle{\bigcup_{\gamma \in \Lambda} \support(y_{\gamma})}$, then $\Gamma \subseteq J$ and $| \Gamma | \leqslant | \Lambda | \cdot \omega < \kappa$ as $\kappa$ is uncountable.\smallskip
	
	We \textit{claim} that for every $\alpha \in \kappa \setminus \Lambda$ there is $j \in \Gamma$ such that $\langle y_{\alpha}, f_j \rangle \neq 0$. Indeed, otherwise there is $\alpha_0 \in \kappa \setminus \Lambda$ such that for all $j \in \Gamma$ we have $\langle y_{\alpha_{0}}, f_j \rangle = 0$. Let $\beta \in \Lambda$, and let $j \in \support(y_{\beta})$. Then $j \in \Gamma$ and hence  we conclude from the above that $\langle y_{\alpha_{0}}, f_j \rangle = 0$, thus $\support(y_{\alpha_{0}}) \cap \support(y_{\beta}) = \varnothing$. Consequently, $\Lambda \subsetneq \Lambda \cup \{\alpha_0 \}$ and $y_{\alpha}, y_{\beta}$ are disjointly supported for any distinct $\alpha, \beta \in \Lambda \cup \{\alpha_0 \}$. This contradicts the maximality of $\Lambda$.\smallskip
	
	Combining the claim with $| \Gamma| \cdot \omega < \kappa$, we obtain that there is $j_0 \in \Gamma$ such that the set
	\[
	S:= \{ \alpha \in \kappa \setminus \Lambda \colon  \Re \langle y_{\alpha}, f_{j_{0}} \rangle  > 0  \}
	\]
	is uncountable. It follows that there is $\varepsilon \in (0,1)$ such that the set
	\[
	S':=\{ \alpha \in \kappa \setminus \Lambda \colon  \Re \langle y_{\alpha}, f_{j_{0}} \rangle  \geqslant  \varepsilon \}
	\]
	is uncountable.
	
	Let $(\alpha_n)_{n =1}^\infty$ be a sequence in $S'$. Then for all $N \in \mathbb{N}$:
	\begin{align}
	\varepsilon \cdot \ln(N+1) &\leqslant \sum_{n=1}^N n^{-1} \Re\langle y_{\alpha_n}, f_{j_0} \rangle  = \Re \left( \sum_{n=1}^N n^{-1} \langle y_{\alpha_n}, f_{j_0} \rangle \right) \leqslant  \left| \sum_{n=1}^N n^{-1} \langle y_{\alpha_n}, f_{j_0} \rangle \right| \notag \\
	&= \left| \Big\langle T \left( \sum_{n=1}^N n^{-1} x_{\alpha_n} \right), f_{j_0} \Big\rangle \right| \leqslant \| T \| \left\| \sum_{n=1}^N n^{-1} x_{\alpha_n} \right\| K.
	\end{align}
	This contradicts the fact that $(x_{\alpha})_{\alpha < \kappa}$ is equivalent to the standard unit vector basis of $c_0(\kappa)$ or $\ell_p(\kappa)$, where $1< p < \infty$. Thus $| \Lambda | = \kappa$ must hold.
\end{proof}

The following corollary is an analogue of Rosenthal's result \cite[Remark~1 on p.~30]{rosenthalmeasure}; it is stated in \cite[Corollary~3.3]{jksch} without a proof. For the convenience of the reader we present the details here.

\begin{Cor}\label{ellprosenthal}
	Suppose that $\lambda, \kappa$ are uncountable cardinals with $\lambda \geqslant  \kappa$ and $p \in (1, \infty)$. Let $(x_{\alpha})_{\alpha < \kappa}$ be a normalised, transfinite sequence in $\ell_p(\lambda)$, which is equivalent to the standard unit vector basis of $\ell_p(\kappa)$. For $T \in \mathscr{B}(\ell_p(\lambda))$, if $\inf \{ \|Tx_{\alpha} \|\colon \alpha < \kappa \} > 0$, then $T \notin \mathscr{S}_{\ell_p(\kappa)}(\ell_p(\lambda))$.
\end{Cor}

\begin{proof}
	Applying Proposition \ref{disjointimage} twice, we can take $\Lambda \subseteq \kappa$ with $| \Lambda | = \kappa$ such that both $(x_{\alpha})_{\alpha \in \Lambda}$ and $(Tx_{\alpha})_{\alpha \in \Lambda}$ consist of disjointly supported vectors. Let $\varepsilon \in (0,1)$ be such that $\| Tx_{\alpha} \| \geqslant \varepsilon$ for each $\alpha < \kappa$. Let $Z:= \overline{\spanning} \{ x_{\alpha} \colon  \alpha \in \Lambda\}$ and take $y \in Z$ arbitrary. Then
	\begin{align}
	\| Ty \|^p &= \left\| T \left( \sum\limits_{\alpha \in \Lambda} y(\alpha) x_{\alpha} \right) \right\|^p = \left\| \sum\limits_{\alpha \in \Lambda} y(\alpha) Tx_{\alpha} \right\|^p = \sum\limits_{\alpha \in \Lambda} \| y(\alpha) Tx_{\alpha} \|^p \notag \\
	&= \sum\limits_{\alpha \in \Lambda} | y(\alpha) |^p \| Tx_{\alpha} \|^p \geqslant \varepsilon^p \sum\limits_{\alpha \in \Lambda} | y(\alpha) |^p = \varepsilon^p \| y \|^p, 
	\end{align}
	hence $T |_Z$ is bounded below. As $Z \cong \ell_p(\kappa)$, the claim follows.
\end{proof}

\begin{Lem}\label{cardbound}
Let $\lambda$, $\kappa$ be cardinal numbers with $\lambda \geqslant \kappa$ and $\cf(\kappa) > \omega$. Consider one of the following cases: 
\begin{itemize}
    \item $E_{\lambda}:= \ell_p(\lambda)$ and $E_{\kappa}:= \ell_p(\kappa)$ for $p \in (1, \infty)$;
    \item $E_{\lambda}:= \ell_{\infty}^c(\lambda)$ and $E_{\kappa}:= c_0(\kappa)$;
    \item $E_{\lambda}:= c_0(\lambda)$ and $E_{\kappa}:= c_0(\kappa)$.
\end{itemize}
Then for any $T \in \mathscr{S}_{E_{\kappa}}(E_{\lambda})$, the cardinality of the set comprising those $\alpha < \lambda$ for which $Te_{\alpha} \neq 0$ is strictly less than $\kappa$.
\end{Lem}

\begin{proof}
	Contrapositively, suppose that the set
	\[
	S:= \{\alpha <\lambda\colon \|Te_{\alpha} \| > 0 \}
	\]
	has cardinality at least $\kappa$. Set $S_n:= \{\alpha <\lambda\colon \|Te_{\alpha} \| \geqslant 1/n \}$ for every $n \in \mathbb{N}$. Then $S = \textstyle{\cup_{n=1}^\infty  S_n}$, thus by Lemma \ref{cofinal} there is $m \in \mathbb{N}$ such that $|S_m| \geqslant \kappa$. We may assume without loss of generality that $|S_m| = \kappa$. Consequently $\inf \{\|Te_{\alpha} \| \colon  \alpha \in S_m \} > 0$. \smallskip
	
	If $E_{\lambda} = \ell_p(\lambda)$ and $E_{\kappa} = \ell_p(\kappa)$ for $p \in (1, \infty)$, then Corollary \ref{ellprosenthal} implies $T \notin \mathscr{S}_{E_{\kappa}}(E_{\lambda})$. \smallskip
	
	If $E_{\lambda}= \ell_{\infty}^c(\lambda)$ and $E_{\kappa}= c_0(\kappa)$, or $E_{\lambda}= c_0(\lambda)$ and $E_{\kappa}= c_0(\kappa)$, then by \cite[Remark~1~on p.~30]{rosenthalmeasure} there is a closed subspace $F$ of $E_{\lambda}$ such that $F \cong c_0(\kappa)$ and $T |_F$ is bounded below. This is equivalent to saying that $T \notin \mathscr{S}_{E_{\kappa}}(E_{\lambda})$.
\end{proof}

We shall need the following result when dealing with the case $p=1$ case in the proof of Theorem \ref{kappasingingen} (\textit{cf.} \cite[first bullet point in the proof of Lemma~3.15]{jksch}). Let us first introduce the following notation:

Let $\lambda$ be an infinite cardinal and let $E_{\lambda}:= c_0(\lambda)$ or $E_{\lambda}:= \ell_{\infty}^c(\lambda)$ or $E_{\lambda} := \ell_p(\lambda)$, where $p \in [1, \infty)$. For $\Lambda \subseteq \lambda$ we define
\begin{align}
(P_{\Lambda}x)(\alpha) := \left\{
\begin{array}{l l}
x(\alpha) & \quad \text{if  } \alpha \in \Lambda \\
0 & \quad \text{otherwise} \\
\end{array} \right. \quad (x \in E_{\lambda}).
\end{align}
Clearly $P_{\Lambda} \in \mathscr{B}(E_{\lambda})$ is an idempotent with $\Ran(P_{\Lambda})$ isometrically isomorphic to $E_{| \Lambda |}$.

\begin{Lem}\label{ell1approx}
	Let $\lambda, \kappa$ be infinite cardinals with $\lambda \geqslant \kappa$. Let $S \in \mathscr{S}_{\ell_1(\kappa)}(\ell_1(\lambda))$. Then for every $\varepsilon \in (0,1)$ there is $\Gamma \subseteq \lambda$ with $|\Gamma | < \kappa$ such that $\| P_{\Gamma} S - S \| \leqslant \varepsilon$.
\end{Lem}

\begin{proof}
	We prove the statement contrapositively. Assume that there is $\varepsilon \in (0,1)$ such that $\| P_{\lambda \setminus \Gamma} S \| = \| P_{\Gamma} S - S \| > \varepsilon$ for every $\Gamma \subseteq \lambda$ with $| \Gamma | < \kappa$. Let us define the sets
	\begin{align}
	\mathscr{Z} &:= \left\{ H \in \mathscr{P}(\ell_1(\lambda) \times \mathscr{P}(\lambda)) \colon  \forall (x, E) \in H\colon  \|x\| =1, \, |E| < \infty, \, \| Sx|_E \| \geqslant \varepsilon/2 \right\} \notag \\
	\mathscr{Y} &:= \left\{ H \in \mathscr{Z} \colon  \forall (x, E), (y,F) \in H\colon  (E \neq F \implies E \cap F = \varnothing) \right\},
	\end{align}
	and consider $\mathscr{Y}$ with the ordering given by set-theoretic containment. It is clear that every chain in $\mathscr{Y}$ has an upper bound in $\mathscr{Y}$, hence by the Kuratowski--Zorn Lemma there is a maximal element $M \in \mathscr{Y}$. We \textit{claim} that $|M| \geqslant \kappa$.\smallskip
	
	Assume towards a contradiction that $|M| < \kappa$. Let $\Gamma:= \bigcup_{(x, E) \in M} E$. Then $\Gamma \subseteq \lambda$ with $|\Gamma| < \kappa$. Indeed, $E$ is a finite subset of $\lambda$ for each $(x,E) \in M$, hence if $\kappa = \omega$ then $\Gamma$ is finite; if $\kappa$ is uncountable then $| \Gamma | \leqslant |M| \cdot \omega < \kappa$. By the assumption $\| P_{\lambda \setminus \Gamma} S \| > \varepsilon$, there is $y \in \ell_1(\lambda)$ with $\|y\| =1$ and $\sum_{\alpha \in \lambda \setminus \Gamma} |(Sy)(\alpha)| = \| P_{\lambda \setminus \Gamma} S y \| > \varepsilon$. As $\support(Sy)$ is countable, there is a finite set $F \subseteq \support(Sy) \cap (\lambda \setminus \Gamma)$ such that $\| Sy |_{F} \| = \sum_{\alpha \in F} |(Sy)(\alpha)| \geqslant \varepsilon /2$. From $F \subseteq \lambda \setminus \Gamma$ we see that $F \cap E = \varnothing$ for each $(x, E) \in M$. Thus $M \subsetneq M \cup \{(y,F)\} \in \mathscr{Y}$, which contradicts the maximality of $M$ in $\mathscr{Y}$. Hence $|M| \geqslant \kappa$ must hold.\smallskip
	
	Let $(x_{\alpha},E_{\alpha})_{\alpha \in \Lambda}$ be a collection in $M$ with $\Lambda \subseteq \lambda$ and $|\Lambda | = \kappa$. As $\| Sx_{\alpha} |_{E_{\alpha}} \| \geqslant \varepsilon /2$ for each $\alpha \in \Lambda$, it follows from \cite[Propositions~3.2 and 3.1]{rosenthalmeasure} that there is $\Lambda' \subseteq \Lambda$ with $| \Lambda' | = | \Lambda | = \kappa$ such that $X:= \overline{\spanning} \{x_{\alpha}\colon  \alpha \in \Lambda' \} \cong \ell_1(\kappa)$ and $S |_X$ is bounded below. This yields $S \notin \mathscr{S}_{\ell_1(\kappa)}(\ell_1(\lambda))$, as required.
\end{proof}

We recall that it is shown in \cite[Theorem~1.3, Proposition~3.9]{jksch} and \cite[Proposition~2.8]{acgjm} that for every uncountable cardinal $\lambda$ the Banach spaces $\ell_{\infty}^c(\lambda)$, $\ell_p(\lambda)$ ($1 \leqslant p <\infty$), and $c_0(\lambda)$ are complementably homogenous. In fact, the following formally stronger results hold, \textit{cf.} \cite[Proposition~3.9 and Remark~3.11]{jksch}:
{
\begin{Prop}\label{complhomextra}
	Let $\lambda, \kappa$ be infinite cardinals with $\lambda \geqslant \kappa$. Consider one of the following cases:
	\begin{itemize}
	    \item $E_{\lambda}:= \ell_p(\lambda)$ and $E_{\kappa}:= \ell_p(\kappa)$ for $p \in [1, \infty)$;
	    \item $E_{\lambda}:= \ell_{\infty}^c(\lambda)$ and $E_{\kappa}:= \ell_{\infty}^c(\kappa)$;
	    \item $E_{\lambda}:= c_0(\lambda)$ and $E_{\kappa}:= c_0(\kappa)$.
	\end{itemize} If $Y$ is a closed subspace of $E_{\lambda}$ with $Y \cong E_{\kappa}$, then there exists a complemented subspace $X$ of $E_{\lambda}$ with $X \subseteq Y$ such that $X \cong E_{\kappa}$. In the latter case, $Y$ is already complemented in $E_\lambda$.
\end{Prop}
\begin{proof}
	We only need to show the statement for $c_0(\lambda)$, the other cases are covered in \cite[Proposition~3.9 and Remark~3.11]{jksch}.\smallskip

Let $Y$ be a closed subspace of $c_0(\lambda)$ such that $Y \cong c_0(\kappa)$. There is a set $\Lambda \subseteq \lambda$ with $| \Lambda | = \kappa$ such that $Y \subseteq c_0({\Lambda})$. As $Y \cong c_0(\kappa) \cong c_0(\Lambda)$, it follows from \cite[Proposition~2.8]{acgjm} that $Y$ is complemented in $c_0(\Lambda)$. As the latter space is complemented in $c_0(\lambda)$, the conclusion follows.
\end{proof}
}

The proposition above implies a convenient corollary. Before we state it, let us remind the reader that it is proved in \cite[Theorem~3.14]{jksch} that for infinite cardinals $\lambda \geqslant \kappa$, the ideals of $\ell_{\infty}^c(\kappa)$-singular and $c_0(\kappa)$-singular operators on $\ell_{\infty}^c(\lambda)$ coincide. This is, $$\mathscr{S}_{\ell_{\infty}^c(\kappa)}(\ell_{\infty}^c(\lambda)) = \mathscr{S}_{c_0(\kappa)}(\ell_{\infty}^c(\lambda)).$$
We shall implicitly use this fact in the subsequent sections.

\begin{Cor}\label{corcomplhom}
		Let $\lambda, \kappa$ be infinite cardinals with $\lambda \geqslant \kappa$. Consider one of the following cases:
	\begin{itemize}
	    \item $E_{\lambda}:= \ell_p(\lambda)$ and $E_{\kappa}:= \ell_p(\kappa)$ for $p \in [1, \infty)$;
	    \item $E_{\lambda}:= \ell_{\infty}^c(\lambda)$ and $E_{\kappa}:= \ell_{\infty}^c(\kappa)$;
	    \item $E_{\lambda}:= c_0(\lambda)$ and $E_{\kappa}:= c_0(\kappa)$.
	\end{itemize}  Let $T \in \mathscr{B}(E_{\lambda})$ be such that $T \notin \mathscr{S}_{E_{\kappa}}(E_{\lambda})$. Then there is a closed subspace $E$ of $E_{\lambda}$ such that $E \cong E_{\kappa}$, $T |_{E}$ is bounded below, and $\Ran(T |_{E})$ is complemented in $E_{\lambda}$.
\end{Cor}

\begin{proof}
	By the hypothesis, there is a closed subspace $E'$ of $E_{\lambda}$ such that $E' \cong E_{\kappa}$ with $T |_{E'}$ bounded below. In particular, $\Ran(T |_{E'}) \cong E' \cong E_{\kappa}$, thus Proposition \ref{complhomextra} yields a complemented subspace $E^{''}$ of $E_{\lambda}$ with $E^{''} \subseteq \Ran(T |_{E'})$ such that $E^{''} \cong E_{\kappa}$. Set $E := E' \cap T^{-1}[E^{''}]$, which is clearly a closed subspace of $E_{\lambda}$. We \textit{claim} that $T |_{E}^{E^{''}} \in \mathscr{B}(E,E^{''})$ is an isomorphism. Clearly $T |_{E}^{E^{''}}$ is injective, in fact it is bounded below, since $E \subseteq E'$ and $T |_{E'}$ is already bounded below. As $E^{''} \subseteq \Ran(T |_{E'})$, the operator $T |_{E}^{E^{''}}$ is surjective. From the claim we conclude that $\Ran(T |_{E}) = \Ran(T |_{E}^{E^{''}}) = E''$ is complemented in $E_{\lambda}$ and isomorphic to $E_{\kappa}$.
\end{proof}

The next result is proved for spaces of the form $\ell_{\infty}^c(\lambda)$ in \cite[Lemma~3.17]{jksch}, however its counterpart for spaces of the form $\ell_p(\lambda)$ neither is explicitly stated nor proved therein, even though it is certainly known to the authors as it is implicitly used in the proof of \cite[Theorem~1.5]{jksch}. For the sake of completeness we include a proof.

\begin{Thm}\label{kappasingingen}
	Let $\lambda, \kappa$ be infinite cardinals with $\lambda \geqslant \kappa$. Consider one of the following cases:
	\begin{itemize}
	    \item $E_{\lambda}:= \ell_p(\lambda)$ and $E_{\kappa}:= \ell_p(\kappa)$ for $p \in [1, \infty)$;
	    \item $E_{\lambda}:= \ell_{\infty}^c(\lambda)$ and $E_{\kappa}:= \ell_{\infty}^c(\kappa)$;
	    \item $E_{\lambda}:= c_0(\lambda)$ and $E_{\kappa}:= c_0(\kappa)$.
	\end{itemize}  Let $T \in \mathscr{B}(E_{\lambda})$ be such that $T \notin \mathscr{S}_{E_{\kappa}}(E_{\lambda})$. Then $\mathscr{S}_{E_{\kappa^+}}(E_{\lambda})$ is contained in the closed, two-sided ideal generated by $T$.
\end{Thm}

\begin{proof}
	The case when $E_{\lambda} = \ell_{\infty}^c(\lambda)$ and $E_{\kappa} = \ell_{\infty}^c(\kappa)$ is  \cite[Lemma~3.17]{jksch}, so we may move on to the remaining cases (so far except the case $p = 1$, which will be treated separately). We split the proof into three parts.
	\begin{romanenumerate}
	    \item Let $S \in \mathscr{S}_{E_{\kappa^+}}(E_{\lambda})$. Consider the set
	\[
	\Lambda := \{\alpha < \lambda\colon  Se_{\alpha} \neq 0\}.
	\]
	As every successor cardinal number is regular, Lemma \ref{cardbound} implies $| \Lambda | < \kappa^+$. Let $\Gamma := \bigcup_{\alpha \in \Lambda} \support (Se_{\alpha})$, clearly $|\Gamma| \leqslant \kappa$. As
	\[
	\bigcup\limits_{x \in E_{\lambda}} \support (Sx)  =  \bigcup\limits_{\alpha \in \Lambda} \support (Se_{\alpha}),
	\]
	it follows from the definition that $P_{\Gamma} S =S$ and $\Ran(P_{\Gamma}) \cong E_{|\Gamma|}$.\smallskip
	
	    \item Since $|\Gamma | \leqslant \kappa$ and $T \notin \mathscr{S}_{E_{\kappa}}(E_{\lambda})$, we have $T \notin \mathscr{S}_{E_{|\Gamma|}}(E_{\lambda})$. By Corollary \ref{corcomplhom} we can take a closed subspace $E$ of $E_{\lambda}$ such that $E \cong E_{|\Gamma|}$, $T |_{E}$ is bounded below, and $\Ran(T |_{E})$ is complemented in $E_{\lambda}$. Clearly $T_1 := T |_{E}^{\Ran(T |_{E})} \in \mathscr{B}(E, \Ran(T |_{E}))$ is an isomorphism. Let $Q \in \mathscr{B}(E_{\lambda})$ be an idempotent such that $\Ran(Q)=\Ran(T |_E)$ and let $\iota \in \mathscr{B}(E, E_{\lambda})$ be the inclusion operator.\smallskip
	    
	    \item As we have $\Ran(P_{\Gamma}) \cong E_{|\Gamma|} \cong E \cong \Ran(T |_{E}) = \Ran(Q)$, we may take an isomorphism $V \in \mathscr{B}(\Ran(P_{\Gamma}), \Ran(Q))$. It is clear that $U:= P_{\Gamma} |_{\Ran(P_{\Gamma})} \circ V^{-1} \circ Q |^{\Ran(Q)} \in \mathscr{B}(E_{\lambda})$. To see that $S$ is contained in the two-sided ideal generated by $T$, it is sufficient to observe that
	\begin{align}
	U \circ T \circ \iota &\circ T_1^{-1} \circ V \circ P_{\Gamma} |^{\Ran(P_{\Gamma})} \circ S = U \circ Q |_{\Ran(Q)} \circ V \circ P_{\Gamma} |^{\Ran(P_{\Gamma})} \circ S \notag \\
	&= P_{\Gamma} |_{\Ran(P_{\Gamma})} \circ V^{-1} \circ Q |^{\Ran(Q)} \circ Q |_{\Ran(Q)} \circ V \circ P_{\Gamma} |^{\Ran(P_{\Gamma})} \circ S \notag \\
	&= P_{\Gamma} |_{\Ran(P_{\Gamma})} \circ V^{-1} \circ V \circ P_{\Gamma} |^{\Ran(P_{\Gamma})} \circ S \notag \\
	&= P_{\Gamma} |_{\Ran(P_{\Gamma})} \circ P_{\Gamma} |^{\Ran(P_{\Gamma})} \circ S \notag \\
	&= P_{\Gamma} \circ S \label{ellpeq} \\
	&= S.
	\end{align}
	\end{romanenumerate}
	
	It remains to show that the theorem holds for the pair $E_{\lambda} = \ell_1(\lambda)$, $E_{\kappa} = \ell_1(\kappa)$. This time we split the argument into two steps (where the latter step roughly corresponds to the last two steps in the previous part of the proof).
	\begin{romanenumerate}
	    \item\label{ell1_i} Let $S \in \mathscr{S}_{\ell_1(\kappa^+)}(\ell_1(\lambda))$. Fix $\varepsilon \in (0,1)$. It follows from Lemma \ref{ell1approx} that we can take $\Gamma \subseteq \lambda$ with $|\Gamma | < \kappa^+$ and $\| P_{\Gamma}S -S \| \leqslant \varepsilon$. Clearly $\Ran(P_{\Gamma}) \cong \ell_1(| \Gamma |)$.\smallskip
	    \item Since $|\Gamma | \leqslant \kappa$ and $T \notin \mathscr{S}_{\ell_1(\kappa)}(\ell_1(\lambda))$, we have $T \notin \mathscr{S}_{\ell_1(|\Gamma|)}(\ell_1(\lambda))$. It follows from Corollary \ref{corcomplhom} that there is a closed subspace $E$ of $\ell_1(\lambda)$ such that $E \cong \ell_1(|\Gamma|)$, $T |_{E}$ is bounded below, and $\Ran(T |_{E})$ is complemented in $\ell_1(\lambda)$. Proceeding exactly as in the $p \in (1,\infty)$ case, we arrive at the corresponding version of equation \eqref{ellpeq}, which shows that $P_{\Gamma} S$ belongs to the (non-closed) algebraic two-sided ideal generated by $T$. Together with \eqref{ell1_i}, this yields that $S$ belongs to the closed, two-sided ideal generated by $T$.\end{romanenumerate}\end{proof}
	
Let us conclude this section with observing that on long sequence spaces $E_{\lambda}$ the ideal of operators with separable range coincides with the ideal of $E_{\omega_1}$-singular operators:\pagebreak

\begin{Lem}\label{seprangeomega1}
Let $\lambda$ be an infinite cardinal number and consider one of the following cases:
\begin{itemize}
    \item $E_{\lambda}:= \ell_p(\lambda)$ and $E_{\omega_1}:= \ell_p(\omega_1)$ for $p \in [1, \infty)$;
    \item $E_{\lambda}:= \ell_{\infty}^c(\lambda)$ and $E_{\omega_1}:= c_0(\omega_1)$;
    \item $E_{\lambda}:= c_0(\lambda)$ and $E_{\omega_1}:= c_0(\omega_1)$.
\end{itemize}
Then $\mathscr{X}(E_{\lambda})= \mathscr{S}_{E_{\omega_1}}(E_{\lambda})$.
\end{Lem}

\begin{proof}
    By the last paragraph of Section \ref{sect: op ids}, the containment $\mathscr{X}(E_{\lambda}) \subseteq \mathscr{S}_{E_{\omega_1}}(E_{\lambda})$ is clear. To see the other direction, suppose $T \in \mathscr{B}(E_{\lambda})$ is such that $T \notin \mathscr{X}(E_{\lambda})$.\smallskip 
    
    Assume first $E_{\lambda}= \ell_1(\lambda)$ and $E_{\omega_1}= \ell_1(\omega_1)$. As $\overline{\Ran(T)}$ is a non-separable, closed subspace of $E_{\lambda}$, it follows from \cite[point~(5)~on~p.~185]{koethe} that there is a~closed (complemented) subspace $W$ of $E_{\lambda}$ such that $W \subseteq \overline{\Ran(T)}$ and $W \cong E_{\omega_1}$. Let us pick a normalised transfinite basic sequence $(w_{\alpha})_{\alpha < \omega_1}$ in $W$ such that it is equivalent to the standard unit vector basis of $E_{\omega_1}$. Hence for each $\alpha < \omega_1$ there is $x_{\alpha} \in E_{\lambda}$ such that $\| w_{\alpha} - T x_{\alpha} \| < 1/2$. It follows from \cite[Example~30.12]{jameson} or \cite[Fact~5.2]{hkr2020} that $(Tx_{\alpha})_{\alpha < \omega_1}$ is a transfinite basic sequence in $E_{\lambda}$ equivalent to $(w_{\alpha})_{\alpha < \omega_1}$, and hence to the standard unit vector basis of $E_{\omega_1}$. In particular, there is $\delta \in (0,1)$ such that $\| Tx_{\alpha} -Tx_{\beta} \| \geqslant \delta$ for each distinct $\alpha, \beta < \omega_1$. Clearly, there is some $n_0 \in \mathbb{N}$ such that the set $\Gamma:= \{ \alpha < \omega_1 \colon \|x_{\alpha} \| \leqslant n_0 \}$ has cardinality $\omega_1$. In conclusion, $(x_{\alpha})_{\alpha \in \Gamma}$ is a bounded transfinite sequence in $E_{\lambda}$ such that $\| Tx_{\alpha} - Tx_{\beta} \| \geqslant \delta$ for each distinct $\alpha, \beta \in \Gamma$, where $| \Gamma| = \omega_1$. Therefore \cite[Corollary~on~p.~29]{rosenthalmeasure} applies; there is a closed (complemented) subspace $Z$ of $E_{\lambda}$ such that $Z \simeq E_{\omega_1}$ and $T |_Z$ is bounded below. Consequently $T \notin \mathscr{S}_{E_{\omega_1}}(E_{\lambda})$.\smallskip
    
    We now consider the remaining cases. As $T$ is continuous, it follows that
    \[
    \Ran(T) \subseteq \overline{\spanning}\{Te_{\alpha} \colon \alpha < \lambda\},
    \]
    hence the right-hand side cannot be separable. This in particular implies that the set $\{\alpha < \lambda \colon Te_{\alpha} \neq 0\}$ must have cardinality at least $\omega_1$, which in turn together with Lemma \ref{cardbound} yields $T \notin \mathscr{S}_{E_{\omega_1}}(E_{\lambda})$. 
\end{proof}

\subsection{An application: $\sigma_{\rm SOT}$-closed ideals}

Let us briefly recall the notion of the strong operator topology on $\mathscr{B}(X)$. If $X$ is a Banach space, then the \textit{strong operator topology $\tau_{\rm SOT}$ on $\mathscr{B}(X)$} is the smallest topology $\tau'$ on $\mathscr{B}(X)$ such that for every $x \in X$ the map
\begin{align}
\varepsilon_x\colon  \mathscr{B}(X) \rightarrow X; \quad T \mapsto Tx
\end{align}
is $\tau'$-to-norm continuous. The topology $\tau_{\rm SOT}$ is a linear, locally convex, Hausdorff topology on $\mathscr{B}(X)$. We say that a net $(T_i)_{i \in I}$ in $\mathscr{B}(X)$ \textit{SOT-converges to $T \in \mathscr{B}(X)$} if $(T_ix)_{i \in I}$ converges to $Tx \in X$ in norm for every $x \in X$. This notion of convergence characterises convergence with respect to the $\tau_{\rm SOT}$ topology, in the sense that a net $(T_i)_{i \in I}$ in $\mathscr{B}(X)$ converges to $T \in \mathscr{B}(X)$ in the $\tau_{\rm SOT}$ topology if and only if $(T_i)_{i \in I}$ SOT-converges to $T$. It follows that a set $C \subseteq \mathscr{B}(X)$ is $\tau_{\rm SOT}$-closed if and only if for any net $(T_i)_{i \in I}$ in $C$ which SOT-converges to some $T \in \mathscr{B}(X)$ it follows that $T \in C$.\smallskip

Let us recall that the $\tau_{\rm SOT}$-closure of the finite-rank operators $\mathscr{F}(X)$ is the whole of $\mathscr{B}(X)$. Indeed, let ${\rm Fin}\, X$ be the set of all finite-dimensional subspaces of $X$. We consider ${\rm Fin}\, X$ ordered by the inclusion. For every $F \in {\rm Fin}\, X$ let us fix an idempotent $P_F \in \mathscr{B}(X)$ with $\Ran(P_F) = F$. Then $(P_F)_{F \in {\rm Fin}\, X}$ converges to $I_X$ in the strong operator topology, as $P_Fx=x$ for each $F \in {\rm Fin}\, X$ and each $x \in F$. Consequently, whenever $\mathscr{S}$ is a~subset of $\mathscr{B}(X)$ with $\mathscr{F}(X) \subseteq \mathscr{S}$, then the $\tau_{\rm SOT}$-closure of $\mathscr{S}$ is the whole of $\mathscr{B}(X)$. In particular, there is no non-trivial, proper, two-sided ideal of $\mathscr{B}(X)$ that is $\tau_{\rm SOT}$-closed.\smallskip

The moral of the argument above is that the strong operator topology $\tau_{\rm SOT}$ is `too weak' for $\mathscr{B}(X)$ to have non-trivial, proper two-sided ideals that are $\tau_{\rm SOT}$-closed at the same time. We are about to introduce a topology on $\mathscr{B}(X)$ which sits naturally between $\tau_{\rm SOT}$ and the topology of convergence in operator norm, denoted by $\tau_{\| \cdot \|_{\rm op}}$. We say that a set $C \subseteq \mathscr{B}(X)$ is \textit{$\sigma$-SOT closed}, whenever for every sequence $(T_n)_{n=1}^\infty $ in $C$ which SOT-converges to some $T \in \mathscr{B}(X)$, $T \in C$ follows. We say that $U \subseteq \mathscr{B}(X)$ is \textit{$\sigma$-SOT open}, whenever for any $T \in U$ and a sequence $(T_n)_{n=1}^\infty $ in $\mathscr{B}(X)$ which SOT-converges to $T$ there is $N \in \mathbb{N}$ such that $T_n \in U$ for each $n \geqslant N$. These notions correspond exactly as expected:
\begin{Lem}\label{sigmasotclosedopen}
Let $X$ be a Banach space. Then $C \subseteq \mathscr{B}(X)$ is $\sigma$-SOT closed if and only if the complement $U$ of $C$ in $\mathscr{B}(X)$ is $\sigma$-SOT open.
\end{Lem}
\begin{proof}
	We prove both directions by way of a contraposition. 
	
	Suppose $U \subseteq \mathscr{B}(X)$ is not $\sigma$-SOT open. Hence there is a $T \in U$ and a sequence $(T_n)_{n=1}^\infty $ in $\mathscr{B}(X)$ which SOT-converges to $T$, but for every $N \in \mathbb{N}$ there is $n \geqslant N$ such that $T_n \notin U$. Hence there is a strictly monotone increasing function $\rho \colon  \mathbb{N} \rightarrow \mathbb{N}$ such that $T_{\rho(n)} \in C:= \mathscr{B}(X) \setminus U$ for each $n \in \mathbb{N}$. As $(T_{\rho(n)})_{n=1}^\infty $ also SOT-converges to $T$ and $T \notin C$, we obtain that $C$ is not $\sigma$-SOT closed.\smallskip
	
	Suppose that $C \subseteq \mathscr{B}(X)$ is not $\sigma$-SOT closed. Hence there is a $T \in U:= \mathscr{B}(X) \setminus C$ and a sequence $(T_n)_{n=1}^\infty $ in $C$ which SOT-converges to $T$. Thus $U$ cannot be $\sigma$-SOT open, as claimed.
\end{proof}

\begin{Prop}\label{sigmasotistop}
Let $X$ be a Banach space. Let $\sigma_{\rm SOT}$ be the collection of all $\sigma$-SOT open subsets of $\mathscr{B}(X)$. Then $\sigma_{\rm SOT}$ is a topology on $\mathscr{B}(X)$.
\end{Prop}

\begin{proof}
It is evident that $\mathscr{B}(X) \in \sigma_{\rm SOT}$. The set $\sigma_{\rm SOT}$ is closed under taking finite intersections: Let $(U_i)_{i=1}^m$ be a finite collection in $\sigma_{\rm SOT}$ and let $U:= \cap_{i=1}^m U_i$. Let $T \in U$ and let $(T_n)_{n=1}^\infty $ be a~sequence in $\mathscr{B}(X)$ such that it SOT-converges to $T$. Fix $i \in \{1, \ldots, m\}$. As $U_i \in \sigma_{\rm SOT}$, there is an $N^{(i)} \in \mathbb{N}$ such that $T_n \in U_i$ for each $n \geqslant N^{(i)}$. Let $N:= \max_{1 \leqslant i \leqslant m} N^{(i)}$, then $T_n \in U_i$ for each $i \in \{1, \ldots, m\}$ and $n \geqslant N$. Hence $T_n \in U$ for all $n \geqslant N$, showing $U \in \sigma_{\rm SOT}$. The set $\sigma_{\rm SOT}$ is closed under taking arbitrary unions: Let $(U_i)_{i \in I}$ be a collection in $\sigma_{\rm SOT}$ and let $U:= \cup_{i \in I} U_i$. Let $T \in U$ and let $(T_n)_{n=1}^\infty $ be a sequence in $\mathscr{B}(X)$ that SOT-converges to $T$. In particular $T \in U_j$ for some $j \in I$, hence by $U_j \in \sigma_{\rm SOT}$ there is an $N \in \mathbb{N}$ such that $T_n \in U_j \subseteq U$ for each $n \geqslant N$. Thus $U \in \sigma_{\rm SOT}$, therefore $\sigma_{\rm SOT}$ is a topology as claimed.
\end{proof}

We remark in passing that the topology $\sigma_{\rm SOT}$ is an instance of the so-called `topology induced by $L^*$-convergence' (see, \emph{e.g.}, \cite{dudley}). Such a topology is automatically $T_1$ but need not be Hausdorff in general.\smallskip

\begin{Rem} The reader may wonder at this point whether the topology $\sigma_{\rm SOT}$ is characterised by SOT-convergent \textit{sequences} in $\mathscr{B}(X)$. The mere fact that results such as Lemma \ref{sigmasotclosedopen} and Proposition \ref{sigmasotistop} hold does not automatically yield this in general.\smallskip

Indeed, consider the space $\mathscr{M}_b[0,1]$ of real-valued, bounded, measurable functions on $[0,1]$. One can define a topology $\sigma_{\rm ae}$ on $\mathscr{M}_b[0,1]$ the following way: A set $C \subseteq \mathscr{M}_b[0,1]$ is a.e.- closed whenever for every sequence $(f_n)_{n=1}^\infty $ in $C$ which converges to some $f \in \mathscr{M}_b[0,1]$ almost everywhere, $f \in C$ follows. A set $U \subseteq \mathscr{M}_b[0,1]$ is a.e.-open if for any $f \in U$ whenever $(f_n)_{n=1}^\infty $ is a sequence in $\mathscr{M}_b[0,1]$ which converges to $f$ almost everywhere, there is an $N \in \mathbb{N}$ such that $f_n \in U$ for each $n \geqslant N$. It is easy to see that the corresponding versions of Lemma \ref{sigmasotclosedopen} and Proposition \ref{sigmasotistop} hold, that is $C \subseteq \mathscr{M}_b[0,1]$ is a.e.-closed if and only if $U:= \mathscr{M}_b[0,1] \setminus C$ is a.e.-open, and the collection of a.e.-open subsets is a $T_1$ topology on $\mathscr{M}_b[0,1]$, which we may denote by $\sigma_{\rm ae}$. However, as it is demonstrated by Ordman's argument from \cite{ordman}, there is a~sequence of functions $(f_n)_{n=1}^{\infty}$ in $\mathscr{M}_b[0,1]$ which converges to $0$ with respect to the topology $\sigma_{\rm ae}$, but it does not converge to $0$ almost everywhere.\smallskip

Fortunately, $\sigma_{\rm SOT}$ does not have this pathological property, as we shall see it from the next lemma.\end{Rem}

\begin{Lem}\label{equivconv}
	Let $X$ be a Banach space, let $T \in \mathscr{B}(X)$ and let $(T_n)_{n=1}^\infty $ be a sequence in $\mathscr{B}(X)$. Then $(T_n)_{n=1}^\infty $ SOT-converges to $T$ if and only if it converges to $T$ in the $\sigma_{\rm SOT}$ topology.
\end{Lem}

\begin{proof}
	It is clear from the definition that strong operator convergence implies convergence in the $\sigma_{\rm SOT}$ topology. We show the other direction by way of a contraposition. Suppose $(T_n)_{n=1}^\infty $ does not SOT-converge to $T$. Hence there is $x \in X$ and $\varepsilon \in (0,1)$ such that for each $N \in \mathbb{N}$ there is $n \geqslant N$ with $\| Tx -T_nx \| \geqslant \varepsilon$. We set $U:= \{S \in \mathscr{B}(X)\colon  \|Sx-Tx \| < \varepsilon\}$. Clearly $T \in U$ and for each $N \in \mathbb{N}$ there is $n \geqslant N$ such that $T_n \notin U$. We \textit{claim} that $U \in \sigma_{\rm SOT}$. Indeed, let $S \in U$ and let $(S_n)_{n=1}^\infty $ be a sequence in $\mathscr{B}(X)$ which SOT-converges to $S$. As $\|Tx -Sx \| < \varepsilon$, for $\delta:= \varepsilon- \|Tx- Sx \| \in (0, 1)$ we can take $M \in \mathbb{N}$ such that $\|Sx -S_nx \| < \delta$ holds whenever $n \geqslant M$. Hence $\|Tx - S_n x \| < \varepsilon$ and thus $S_n \in U$ for each $n \geqslant M$, as claimed.
\end{proof}

The second part of the next proposition is a consequence of the fact that the weak*-topology of a Banach space $X^*$ is sequential (or even Fr\'echet--Urysohn) if and only if $X$ is finite-dimensional (see \textit{e.g.} \cite[Proposition~2.6.12]{Megginson}). Since our argument relies on the details of the proof of the aforementioned fact, we present it here in its entirety.

\begin{Prop}
    Let $X$ be a Banach space. Then $\sigma_{\rm SOT}$ is a Hausdorff topology on $\mathscr{B}(X)$ with
    $$\tau_{\rm SOT} \subseteq \sigma_{\rm SOT} \subseteq \tau_{\| \cdot \|_{\rm op}}.$$
    Both inclusions are proper if and only if $X$ is infinite-dimensional.
\end{Prop}

\begin{proof}
In Proposition \ref{sigmasotistop} we saw that $\sigma_{\rm SOT}$ is indeed a topology on $\mathscr{B}(X)$. The containment $\tau_{\rm SOT} \subseteq \sigma_{\rm SOT}$ is evident. Let us then show that $\sigma_{\rm SOT} \subseteq \tau_{\| \cdot \|_{\rm op}}$. Let $C \subseteq \mathscr{B}(X)$ be a $\sigma_{\rm SOT}$-closed subset. In view of Lemma \ref{sigmasotclosedopen} it is enough to show that $C$ is closed in the operator norm. This is immediate: Let $(T_n)_{n=1}^\infty $ be a sequence in $C$ which converges to some $T \in \mathscr{B}(X)$ in the operator norm. Then $(T_n)_{n=1}^\infty $ clearly SOT-converges to $T$, hence $T\in C$ and thus $C$ is closed in the operator norm. Lastly, $\sigma_{\rm SOT}$ is Hausdorff plainly because $\tau_{\rm SOT} \subseteq \sigma_{\rm SOT}$ holds and $\tau_{\rm SOT}$ itself is Hausdorff. \smallskip

Assume now that $X$ is infinite-dimensional, we show that $\sigma_{\rm SOT}$ differs from both $\tau_{\rm SOT}$ and $\tau_{\| \cdot \|_{\rm op}}$.\smallskip

\textit{$(1)$ We first show $\sigma_{\rm SOT} \subsetneq \tau_{\| \cdot \|_{\rm op}}$.} As $X$ is infinite-dimensional, by the Josefson--Nissenzweig Theorem (\cite[Theorem~3.27]{hajek}), we may take a normalised sequence $(f_n)_{n=1}^\infty $ in $X^*$ which converges to $0$ in the $\text{weak}^*$-topology. Let $x_0 \in X$ be a unit vector, and define $T_n:=x_0 \otimes f_n$ for each $n \in \mathbb{N}$. As $T_nz= \langle z, f_n \rangle x_0$ for each $z \in X$ and $n \in \mathbb{N}$, it readily follows that $(T_n)_{n=1}^\infty $ SOT-converges to $0 \in \mathscr{B}(X)$, which by Lemma \ref{equivconv} means that it converges to $0$ in the $\sigma_{\rm SOT}$ topology. As $\|T_n \|=1$ for all $n \in \mathbb{N}$ it follows that $(T_n)_{n=1}^\infty $ cannot be convergent in the operator norm, and hence $\sigma_{\rm SOT} \subsetneq \tau_{\| \cdot \|_{\rm op}}$.\smallskip

\textit{$(2)$ We now show $\tau_{\rm SOT} \subsetneq \sigma_{\rm SOT}$.} As $X$ is infinite-dimensional, so is $X^*$, hence for each $k \in \mathbb{N}$ we may pick a subspace $Y_k$ of $X^*$ with $\dim Y_k = k$. Moreover, by compactness, we may choose a finite $(1/k)$-net $S^{(k)}$ of the sphere $\{h\in Y_k\colon \|h\|=k\}$ in $Y_k$. Let $S:= \bigcup_{k =1}^{\infty} S^{(k)}$. We \textit{claim} that $0 \in X^*$ is in the weak*-closure of $S$. 

Indeed, let $U$ be a weak*-open neighbourhood of $0 \in X^*$. We can take $x_1, \ldots, x_n\in X$ and $\varepsilon > 0$ such that
$$\{f\in X^*\colon \max_{1 \leqslant i \leqslant n} |\langle x_i, f\rangle| < \varepsilon \}\subseteq U.$$
Let $C:= \max_{1 \leqslant i \leqslant n} \|x_i \|$, and take $k \in \mathbb{N}$ sufficiently large with $C/k <\varepsilon$ and $n<k$. Hence there is $g\in Y_k$ such that $\langle x_i, g\rangle = 0$ for each $1 \leqslant i \leqslant n$. We may assume without loss of generality that $\|g\|=k$. We can pick $f \in S^{(k)}$ so that $\|f - g\| \leqslant 1/k$. Consequently,
$$|\langle x_i, f\rangle| = |\langle x_i, f - g\rangle| \leqslant \| x_i \| \|f - g\|\leqslant C/k <\varepsilon \quad (1 \leqslant i \leqslant n),$$
hence $f\in U$. This shows $U \cap S \neq \varnothing$ and thus the claim follows. \smallskip

However, no sequence from $S$ can converge to $0 \in X^*$ in the weak*-topology. For assume towards a contradiction that $(f_n)_{n=1}^{\infty}$ is a sequence in $S$ which converges to $0 \in X^*$ in the weak*-topology. Then $(f_n)_{n=1}^{\infty}$ must be bounded by the Banach--Steinhaus Theorem, hence by the definition of $S$ there is $i_0 \in \mathbb{N}$ such that the set $\{n \in \mathbb{N} \colon f_n \in S^{(i_0)} \}$ is infinite. Thus we may choose a subsequence $(f_n')_{n=1}^{\infty}$ of $(f_n)_{n=1}^{\infty}$ such that $f_n' \in S^{(i_0)}$ for each $n \in \mathbb{N}$. But $S^{(i_0)}$ is finite so $(f'_n)_{n =1}^{\infty}$ must be eventually constant, which contradicts the fact that it converges to $0 \in X^*$ in the weak*-topology. \smallskip

Now, let us fix a unit vector $x_0\in X$. It is clear from the above that $0 \in \mathscr{B}(X)$ belongs to the $\tau_{\rm SOT}$-closure but not to the $\sigma_{\rm SOT}$-closure of the set $\{ x_0 \otimes f \colon f \in S \}$, and hence $\tau_{\rm SOT} \subsetneq \sigma_{\rm SOT}$.
\end{proof}

From now on we may (and do) interchangeably say that a sequence $(T_n)_{n=1}^\infty $ in $\mathscr{B}(X)$ `SOT-converges' or `converges in the $\sigma_{\rm SOT}$ topology' or even `converges in the strong operator topology' to some $T \in \mathscr{B}(X)$.

We recall that the ideal of finite-rank operators $\mathscr{F}(X)$ is $\tau_{\rm SOT}$-dense in $\mathscr{B}(X)$ for every Banach space $X$. This puts the following observation into context.

\begin{Rem}\label{classicsigmasot}
Let $X$ be a Banach space. Consider the following statements:
\begin{romanenumerate}
    \item\label{obs1} $X$ has a Schauder basis;
    \item\label{obs2} $\mathscr{F}(X)$ is $\sigma_{\rm SOT}$-dense in $\mathscr{B}(X)$;
    \item\label{obs3} $X$ has the bounded approximation property.
\end{romanenumerate}
Then \eqref{obs1} $ \Rightarrow$ \eqref{obs2} $\Rightarrow$ \eqref{obs3}.\smallskip

 The implication \eqref{obs1} $ \Rightarrow$ \eqref{obs2} is elementary, as the sequence of coordinate projections SOT-converges to $I_X$.\smallskip

To see \eqref{obs2} $\Rightarrow$ \eqref{obs3}, let us note first that by the hypothesis there is a sequence of finite-rank operators $(T_n)_{n=1}^\infty $ that SOT-converges to the identity operator $I_X$. By the Banach--Steinhaus Theorem, $M:= \sup_{n\in \mathbb N}  \| T_n \| < \infty$, hence \cite[Proposition~4.3]{Ryan} implies that $X$ has the bounded approximation property with constant $M$.
\end{Rem}

In particular, Remark~\ref{classicsigmasot} immediately shows that whenever $X$ has a Schauder basis, there is no non-trivial, proper, two-sided ideal of $\mathscr{B}(X)$ which is $\sigma_{\rm SOT}$-closed. We can, however, always find such ideals whenever $X$ is non-separable:

\begin{Lem}\label{seprangesotclosed}
The norm-closed, two-sided ideal $\mathscr{X}(X)$ of operators with separable range is $\sigma_{\rm SOT}$-closed for any Banach space $X$.
\end{Lem}

\begin{proof}
Let $(T_n)_{n=1}^\infty $ be a sequence in $\mathscr{X}(X)$ which SOT-converges to some $T \in \mathscr{B}(X)$. This immediately yields $\Ran(T) \subseteq \overline{\bigcup_{n=1}^\infty  \Ran(T_n)}$, where the closure is taken with respect to the norm topology of $X$. As $\Ran(T_n)$ is separable for each $n \in \mathbb{N}$, the claim readily follows.
\end{proof}

We note that even if $\mathscr{S}_X(X)$ is the unique maximal ideal of $\mathscr{B}(X)$ for some Banach space $X$, it may or may not be $\sigma_{\rm SOT}$-closed:

\begin{Rem}\label{classicalseqsotclosed}
Let $X:= c_0$ or $X:= \ell_p$, where $1 \leqslant p < \infty$. Then $\mathscr{S}_X(X) = \mathscr{K}(X)$ is the unique maximal ideal of $\mathscr{B}(X)$ by Corollary \ref{classicalideal}, but it cannot be $\sigma_{\rm SOT}$-closed by Remark \ref{classicsigmasot} as $X$ has a Schauder basis. Let $X:= \ell_{\infty}$, then $\mathscr{S}_X(X) = \mathscr{X}(X)$ is the unique maximal ideal of $\mathscr{B}(X)$ by Corollary \ref{classicalideal}, and it is $\sigma_{\rm SOT}$-closed by Lemma \ref{seprangesotclosed}.
\end{Rem}

We are now ready to prove Theorem B.
\begin{proof}[Proof of Theorem B]
	Let $(T_n)_{n=1}^\infty $ be a sequence in $\mathscr{S}_{E_{\kappa}}(X)$ which converges to $T \in \mathscr{B}(X)$ in the strong operator topology.\smallskip
	
	Assume towards a contradiction that $T \notin \mathscr{S}_{E_{\kappa}}(X)$. Then there is a closed subspace $Z$ of $X$ with $Z \cong E_{\kappa}$ such that $T |_Z$ is bounded below by, say, $\varepsilon \in (0,1)$. Let $(z_{\alpha})_{\alpha< \kappa}$ be a normalised transfinite sequence in $Z$ equivalent to the standard unit vector basis of $E_{\kappa}$. By the linear independence, we may set
	\begin{align}
	x^{(s)}_{\alpha, \beta}:= \dfrac{z_{\alpha} - s z_{\beta}}{ \| z_{\alpha} - s z_{\beta} \|} \in Z \quad \left(\{ \alpha, \beta \} \in [\kappa]^2, \; s \in \mathbb{C} \right).
	\end{align}%
	In particular, $\|Tx_{\alpha, \beta}^{(s)} \| \geqslant \varepsilon$ for all $\{ \alpha, \beta \} \in [\kappa]^2$ and $s \in \mathbb{C}$. For each $s \in \mathbb{C}$ and $N \in \mathbb{N}$, we define
	\begin{align}
	\Lambda^{(s)}_N:= \left\{ \{ \alpha, \beta \} \in [\kappa]^2 \colon  (\forall n \geqslant N) \left(\| Tx^{(s)}_{\alpha, \beta} - T_n x^{(s)}_{\alpha, \beta} \| < \varepsilon/2 \right) \right\}.
	\end{align}
	We \textit{claim} that for each $s \in \mathbb{C}$ there is $N \in \mathbb{N}$ with $ | \Lambda^{(s)}_N | = \kappa$. For assume towards a~contradiction that there is $s \in \mathbb{C}$ such that $| \Lambda^{(s)}_N | < \kappa$ for all $N \in \mathbb{N}$, then $| \bigcup_{n=1}^\infty  \Lambda^{(s)}_N | < \kappa$ by Lemma \ref{cofinal}. This is equivalent to saying that the set
	\begin{align}
	\left\{ \{ \alpha, \beta \} \in [\kappa]^2 \colon  (\exists N \in \mathbb{N})(\forall n \geqslant N) \left(\| Tx^{(s)}_{\alpha, \beta} - T_n x^{(s)}_{\alpha, \beta} \| < \varepsilon/2 \right) \right\}
	\end{align}
	has cardinality strictly less than $\kappa$. However, this is impossible. Indeed, for each $\{ \alpha, \beta \} \in [\kappa]^2$ the sequence $(T_n x^{(s)}_{\alpha, \beta})_{n=1}^\infty $ converges to $Tx^{(s)}_{\alpha, \beta} \in X$, hence there is $N \in \mathbb{N}$ such that $\| Tx^{(s)}_{\alpha, \beta} - T_n x^{(s)}_{\alpha, \beta} \| < \varepsilon /2$ for every $n \geqslant N$. This shows the claim.
	
	Fix $s \in \mathbb{C}$ and let us take $N \in \mathbb{N}$ with $| \Lambda^{(s)}_N | = \kappa$. Let $n \geqslant N$ be fixed. Then
	\begin{align}
	\varepsilon /2 = \varepsilon - \varepsilon /2 \leqslant \| T x^{(s)}_{\alpha, \beta} \| - \| Tx^{(s)}_{\alpha, \beta} - T_n x^{(s)}_{\alpha, \beta} \| \leqslant \| T_n x^{(s)}_{\alpha, \beta} \| \quad \left(\{ \alpha, \beta \} \in \Lambda^{(s)}_N \right).
	\end{align}
	We split the rest of the proof into cases.
	\begin{itemize}
		\item Suppose we are in case (1) or (2). We set the parameter $s:=0$, then simply $x^{(s)}_{\alpha, \beta} = z_{\alpha}$ for each $\{\alpha, \beta \} \in [\kappa]^2$. So there is $\Lambda \subseteq \kappa$ with $| \Lambda | = \kappa$ such that $\varepsilon /2 \leqslant \| T_n z_{\alpha} \|$ for each $\alpha \in \Lambda$.
		\begin{itemize}
			\item Suppose we are in the case of (1), that is, $E_{\kappa}= c_0(\kappa)$ and $X$ has an M-basis. Then Proposition \ref{disjointimage} yields that there is $\Gamma \subseteq \Lambda$ with $| \Gamma | = \kappa$ such that $(T_n z_{\alpha})_{\alpha \in \Gamma}$ consists of disjointly supported vectors. Then by \cite[Remark~1 on p.~30]{rosenthalmeasure}, there is a closed subspace $Y$ of $X$ such that $Y \cong c_0(\Gamma) \cong c_0(\kappa)$ and $T_n |_{Y}$ is bounded below. Hence $T_n \notin \mathscr{S}_{c_0(\kappa)}(X)$, a contradiction.
			\item Suppose we are in the case of (2), that is, $E_{\kappa}= \ell_p(\kappa)$ and $X= \ell_p(\lambda)$ where $\lambda \geqslant \kappa$ and $p \in (1, \infty)$. Then by Corollary \ref{ellprosenthal} we have $T_n \notin \mathscr{S}_{\ell_p(\kappa)}(\ell_p(\lambda))$, a~contradiction.
		\end{itemize}\smallskip
		\item Suppose we are in case (3), that is, $E_{\kappa}= \ell_1(\kappa)$ and $X$ is any Banach space. We set the parameter $s:=1$, then $x_{\alpha, \beta}^{(s)} = (z_{\alpha} - z_{\beta}) \| z_{\alpha} - z_{\beta} \|^{-1}$. As $(z_{\alpha})_{\alpha < \kappa}$ is equivalent to the standard unit vector basis of $\ell_1(\kappa)$, there is $\delta \in (0,1)$ such that $\delta \leqslant \| z_{\alpha} - z_{\beta} \|$ for each distinct $\alpha, \beta < \kappa$. Consequently 
		\begin{align}
		\varepsilon /2 &\leqslant \| T_n x^{(s)}_{\alpha, \beta} \| = \| z_{\alpha} - z_{\beta} \|^{-1} \| T_n z_{\alpha} -T_n z_{\beta} \| \notag \\
		&\leqslant \delta^{-1} \| T_n z_{\alpha} -T_n z_{\beta} \| \quad \left(\{ \alpha, \beta \} \in \Lambda^{(s)}_N \right),
		\end{align}
		hence $\varepsilon \delta /2 \leqslant \| T_n z_{\alpha} -T_n z_{\beta} \|$. Thus \cite[Corollary~on~p.~29]{rosenthalmeasure} implies that there is a~closed (complemented) subspace $Y$ of $X$ with $Y \cong \ell_1(\Lambda^{(s)}_N) \cong \ell_1(\kappa)$ such that $T_n |_Y$ is bounded below. Hence $T_n \notin \mathscr{S}_{\ell_1(\kappa)}(X)$, a contradiction.
	\end{itemize}
\end{proof}

One might wonder whether Theorem B holds for \textit{all} uncountable cardinals. We shall see in Lemma \ref{notsigmasotclosed} that this is not the case. To demonstrate this, we will use the following auxiliary lemma: \smallskip

\begin{Lem}\label{projinideal}
    Let $\lambda, \kappa$ be infinite cardinals with $\lambda \geqslant \kappa$, and let $\Lambda \subseteq \lambda$. Consider one of the following cases:
	\begin{itemize}
	    \item $E_{\lambda}:= \ell_p(\lambda)$ and $E_{\kappa}:= \ell_p(\kappa)$ for $p \in [1, \infty)$;
	    \item $E_{\lambda}:= \ell_{\infty}^c(\lambda)$ and $E_{\kappa}:= \ell_{\infty}^c(\kappa)$;
	    \item $E_{\lambda}:= c_0(\lambda)$ and $E_{\kappa}:= c_0(\kappa)$.
	\end{itemize}
	Then $P_{\Lambda} \in \mathscr{S}_{E_{\kappa}}(E_{\lambda})$ if and only if $| \Lambda | < \kappa$.
\end{Lem}

\begin{proof}
We prove both directions by way of a contraposition.

	Suppose $|\Lambda| \geqslant \kappa$. Let $E:= \Ran(P_{\Lambda}) \cong E_{|\Lambda|}$, clearly $P_{\Lambda} |_{E}$ is bounded below. Hence $P_{\Lambda} \notin \mathscr{S}_{E_{|\Lambda|}}(E_{\lambda})$, thus in particular $P_{\Lambda} \notin \mathscr{S}_{E_{\kappa}}(E_{\lambda})$.
	
	Suppose that $P_{\Lambda} \notin \mathscr{S}_{E_{\kappa}}(E_{\lambda})$. Then there is a closed subspace $E$ of $E_{\lambda}$ with $E \cong E_{\kappa}$ such that $P_{\Lambda} |_E$ is bounded below. Thus
\begin{align}
E_{\kappa} \cong E \cong \Ran(P_{\Lambda} |_E) \subseteq \Ran(P_{\Lambda}) \cong E_{|\Lambda|},
\end{align}
consequently $E_{\kappa}$ embeds into $E_{|\Lambda|}$. Hence $|\Lambda | \geqslant \kappa$ must hold.
\end{proof}

\begin{Lem}\label{notsigmasotclosed}
	Let $\lambda, \kappa$ be infinite cardinals such that $\lambda \geqslant \kappa$ and $\cf(\kappa)= \omega$.
	Consider one of the following cases:
	\begin{itemize}
	    \item $E_{\lambda}:= \ell_p(\lambda)$ and $E_{\kappa}:= \ell_p(\kappa)$ for $p \in [1, \infty)$;
	    \item $E_{\lambda}:= c_0(\lambda)$ and $E_{\kappa}:= c_0(\kappa)$.
	\end{itemize}
	The norm closed, two-sided ideal $\mathscr{S}_{E_{\kappa}}(E_{\lambda})$ is not $\sigma_{\rm SOT}$-closed.
\end{Lem}

\begin{proof}
	We only show the $E_{\lambda}= \ell_p(\lambda)$, $E_{\kappa}= \ell_p(\kappa)$-case, the other case follows from an entirely analogous argument. As $\cf(\kappa) = \omega$, we can take a sequence of (regular) cardinals $(\kappa_n)_{n=1}^\infty $ such that $\kappa = \lim_{n\to \infty} \kappa_n$ and $\kappa_n < \kappa_{n+1} < \kappa$ for each $n \in \mathbb{N}$. (If $\kappa = \omega$ then we set $\kappa_n := n$, if $\kappa$ is uncountable then $\kappa_n$ can be chosen uncountable for each $n \in \mathbb{N}$.) We \textit{claim} that $(P_{\kappa_n})_{n=1}^\infty $ converges to $P_{\kappa} \in \mathscr{B}(E_{\lambda})$ in the strong operator topology. Let us note that this is sufficient to prove the lemma. Indeed, $P_{\kappa} \notin \mathscr{S}_{E_{\kappa}}(E_{\lambda})$ and $P_{\kappa_n} \in \mathscr{S}_{E_{\kappa}}(E_{\lambda})$ for each $n \in \mathbb{N}$ by Lemma \ref{projinideal}. In order to show the claim, let us fix an $x \in E_{\lambda}$ and $\varepsilon \in (0,1)$. Take a finite set $F \subseteq \lambda$ such that $\sum_{\alpha \in \lambda \setminus F} |x(\alpha)|^p < \varepsilon^p$. As $\kappa = \lim_{n\to \infty}\kappa_n$ and $F$ is finite, there is $N \in \mathbb{N}$ such that $F \cap \kappa \subseteq \kappa_N$. Consequently
	\begin{align}
	\left\| P_{\kappa}x - P_{\kappa_n}x \right\|^p = \|P_{\kappa \setminus \kappa_n} x \|^p = \sum\limits_{\alpha \in \kappa \setminus \kappa_n} |x(\alpha)|^p \leqslant \sum\limits_{\alpha \in \lambda \setminus F} |x(\alpha)|^p < \varepsilon^p \quad (n \geqslant N),
	\end{align}
	which concludes the claim.\end{proof}

We leave open the question of whether the above lemma holds for $E_{\lambda} =\ell_{\infty}^c(\lambda)$ and $E_{\kappa} = \ell_{\infty}^c(\kappa)$ when $\lambda$ is uncountable and $\lambda \geqslant \kappa$ with $\cf(\kappa)= \omega$. Nonetheless, we make the following remark:

\begin{Rem}
We note that the proof of Lemma \ref{notsigmasotclosed} does not carry over to the $\ell_{\infty}^c$-case. Indeed, let $\lambda$ be an uncountable cardinal and let $\kappa$ be a cardinal with $\lambda \geqslant \kappa$ and $\cf(\kappa)= \omega$. Let $E_{\lambda}:= \ell_{\infty}^c(\lambda)$ and $E_{\kappa}:= \ell_{\infty}^c(\kappa)$. Take a sequence of cardinals $(\kappa_n)_{n=1}^\infty $ with $\kappa = \lim_{n\to\infty} \kappa_n$ and $\kappa_n < \kappa_{n+1} < \kappa$ for each $n \in \mathbb{N}$. We \textit{claim} that $(P_{\kappa_n})_{n=1}^\infty $ does not $\sigma_{\rm SOT}$-converge to $P_{\kappa} \in \mathscr{B}(E_{\lambda})$. To see this we first note that $\mathds{1}_C \in E_{\lambda}$ for each countable $C \subseteq \lambda$. Hence for any countable set $C \subseteq \lambda$ and any $n \in \mathbb{N}$,
\begin{align}
(P_{\kappa} \mathds{1}_C - P_{\kappa_n} \mathds{1}_C)(\alpha) = (P_{\kappa \setminus \kappa_n} \mathds{1}_C)(\alpha) = \left\{
\begin{array}{l l}
1 & \quad \text{if  } \alpha \in C \cap (\kappa \setminus \kappa_n) \\
0 & \quad \text{otherwise} \\
\end{array} \right. \quad (\alpha<\lambda).
\end{align}
Assume towards a contradiction that $(P_{\kappa_n})_{n=1}^\infty $ does $\sigma_{\rm SOT}$-converge to $P_{\kappa} \in \mathscr{B}(E_{\lambda})$. Then from the above we conclude that for each countable $C \subseteq \lambda$ there is $N \in \mathbb{N}$ such that $C \cap (\kappa \setminus \kappa_n) = \varnothing$ whenever $n \geqslant N$. Let $C:= \{ \kappa_m \colon  m \in \mathbb{N} \}$, then there is $N \in \mathbb{N}$ such that $C \subseteq (\lambda \setminus \kappa) \cup \kappa_N$. This is clearly impossible as $ \kappa_n \in \kappa$ and $\kappa_n \notin \kappa_N$ for each $n > N$.
\end{Rem}

Theorem B, Remark \ref{classicalseqsotclosed}, and Lemma \ref{notsigmasotclosed} yield together a characterisation of $\sigma_{\rm SOT}$-closedness of ideals of the form $\mathscr{S}_{E_{\kappa}}(E_{\lambda})$.

\begin{Cor}
	Let $\lambda, \kappa$ be infinite cardinals such that $\lambda \geqslant \kappa$. Consider one of the following cases:
	\begin{itemize}
	    \item $E_{\lambda}:= \ell_p(\lambda)$ and $E_{\kappa}:= \ell_p(\kappa)$ for $p \in [1, \infty)$;
	    \item $E_{\lambda}:= c_0(\lambda)$ and $E_{\kappa}:= c_0(\kappa)$.
	\end{itemize}
	The norm-closed, two-sided ideal $\mathscr{S}_{E_{\kappa}}(E_{\lambda})$ is $\sigma_{\rm SOT}$-closed if and only if $\cf(\kappa) > \omega$.
\end{Cor}

To conclude this section, we demonstrate how the topology $\sigma_{\rm SOT}$ may be used to gain `algebraic' information about ideals of $\mathscr{B}(E_{\lambda})$.

\begin{Prop}
	Let $\lambda,\kappa$ be infinite cardinals with $\lambda \geqslant \kappa$. Consider one of the following cases:
	\begin{itemize}
	    \item $E_{\lambda}:= \ell_p(\lambda)$ and $E_{\kappa}:= \ell_p(\kappa)$ for $p \in [1, \infty)$;
	    \item $E_{\lambda}:= c_0(\lambda)$ and $E_{\kappa}:= c_0(\kappa)$.
	\end{itemize}
	If $\cf(\kappa) = \omega$, then $\mathscr{S}_{E_{\kappa^+}}(E_{\lambda})$ is singly generated as a (norm-)closed, two-sided ideal.
\end{Prop}

\begin{proof}
	By Lemma \ref{notsigmasotclosed}, the ideal $\mathscr{S}_{E_{\kappa}}(E_{\lambda})$ is not $\sigma_{\rm SOT}$-closed, hence there is a sequence $(T_n)_{n=1}^\infty $ in $\mathscr{S}_{E_{\kappa}}(E_{\lambda})$ which converges to some $T \in \mathscr{B}(E_{\lambda})$ in the strong operator topology, where $T \notin \mathscr{S}_{E_{\kappa}}(E_{\lambda})$. On the one hand Theorem \ref{kappasingingen} implies that $\mathscr{S}_{E_{\kappa^+}}(E_{\lambda})$ is contained in the closed, two-sided ideal generated by $T$. On the other hand $\cf(\kappa^+)= \kappa^+ > \omega$ and hence $\mathscr{S}_{E_{\kappa^+}}(E_{\lambda})$ is $\sigma_{\rm SOT}$-closed by Theorem B, therefore $T \in \mathscr{S}_{E_{\kappa^+}}(E_{\lambda})$. Thus $\mathscr{S}_{E_{\kappa^+}}(E_{\lambda})$ and the closed, two-sided ideal generated by $T$ must coincide.
\end{proof}

The proposition above should be compared with \cite[Proposition~4.3]{dawsideal}. In the said result, it is shown that any element of $\mathscr{X}(E_{\lambda})$ generates $\mathscr{X}(E_{\lambda})$ as a closed, two-sided ideal (here $E_{\lambda}= c_0(\lambda)$ or $E_{\lambda}= \ell_p(\lambda)$, where $\lambda$ is any uncountable cardinal and  $p \in [1, \infty)$.) It should be noted that $\mathscr{X}(E_{\lambda})= \mathscr{S}_{E_{\omega_1}}(E_{\lambda})$ for $E_{\lambda}= c_0(\lambda)$ and $E_{\lambda}= \ell_p(\lambda)$ for $p \in [1, \infty)$, by Lemma \ref{seprangeomega1}.

\subsection{The SHAI property of long sequence spaces}

We recall that (a slightly more general version of) the following result was proved in \cite[Lemma~2.6]{horvath2}.

\begin{Lem}\label{noninjbigkernel}
Let $X$ and $Y$ be non-zero Banach spaces, and let $\psi\colon  \mathscr{B}(X) \rightarrow \mathscr{B}(Y)$ be a~surjective, non-injective algebra homomorphism. Then $$\mathscr{E}(X) \subseteq \Ker(\psi).$$
\end{Lem}
We might wonder what other ideals the kernel of a surjective, non-injective algebra homomorphism must contain. Let us recall the following standard terminology. { If $X$ and $W$ are Banach spaces, then the set
\[
\overline{\mathscr{G}}_{W}(X):= \overline{\spanning}\big\{ST \colon T \in \mathscr{B}(X,W), S \in \mathscr{B}(W,X) \big\}
\]
is a closed, two-sided ideal of $\mathscr{B}(X)$ and it is called the \textit{ideal of operators that approximately factor through $W$}. In particular, if $X$ has a complemented subspace isomorphic to $W$, and $P \in \mathscr{B}(X)$ is an idempotent with $\Ran(P) \cong W$ then $\overline{\mathscr{G}}_{W}(X)$ coincides with the closed, two-sided ideal generated by $P$.} 

\begin{Prop}\label{complshai}
	Let $X$ be a Banach space and suppose that $W$ is a non-zero, complemented subspace of $X$ such that $W$ has the SHAI property. Let $Y$ be a non-zero Banach space and let $\psi\colon  \mathscr{B}(X) \rightarrow \mathscr{B}(Y)$ be a surjective, non-injective algebra homomorphism. Then $$\overline{\mathscr{G}}_{W}(X) \subseteq \Ker (\psi).$$
\end{Prop}

\begin{proof}
	 Let $P \in \mathscr{B}(X)$ be an idempotent with $W = \Ran(P)$.\smallskip
	
	Let us observe that in order to prove the proposition it is enough to show that $P \in \Ker(\psi)$. Indeed; if this holds then $\overline{\mathscr{G}}_{W}(X) \subseteq \Ker (\psi)$ by definition, as $\Ker(\psi)$ is a closed, two-sided ideal of $\mathscr{B}(X)$.\smallskip
	
	Assume in search of a contradiction that $P \notin \Ker(\psi)$. Then $Z:= \Ran(\psi(P))$ is a non-zero, closed (complemented) subspace of $Y$. Let us fix $T \in \mathscr{B}(W)$, we observe that $$\psi(P \vert_W \circ T \circ P \vert^W) |_Z^Z \in \mathscr{B}(Z).$$ The only thing we need to check is that the range of $\psi(P \vert_W T P \vert^W) \vert_Z$ is contained in $Z$ which is clearly true since $\psi(P) \psi(P \vert_W T P \vert^W) \psi(P) = \psi(P \vert_W T P \vert^W)$. Consequently the map
	\begin{align}
	\theta\colon  \mathscr{B}(W) \rightarrow \mathscr{B}(Z); \quad T \mapsto \psi(P \vert_W \circ T \circ P \vert^W) \vert_Z^Z
	\end{align}
	is well-defined. It is immediate to see that $\theta$ is a linear map. To see that it is multiplicative, it is enough to recall that $P \vert^W P \vert_W = I_W$, thus by multiplicativity of $\psi$, we obtain $\theta
	(T) \theta(S) = \theta(TS)$ for any $T,S \in \mathscr{B}(W)$.\smallskip
	
	We show that $\theta$ is surjective. To see this we fix $R \in \mathscr{B}(Z)$. Then $\psi(P) \vert_Z R \psi(P) \vert^Z \in \mathscr{B}(Y)$ so by surjectivity of $\psi$ it follows that there exists $A \in \mathscr{B}(X)$ such that $\psi(A) = \psi(P) \vert_Z R \psi(P) \vert^Z$. Consequently $\psi(PAP) = \psi(P) \psi(A) \psi(P) = \psi(P) \vert_Z R \psi(P) \vert^Z$ and thus by the definition of $\theta$ we obtain 
	\begin{align}
	\theta(P \vert^W \circ A \circ P \vert_W) &=  \psi(P \vert_W \circ P \vert^W \circ A \circ P \vert_W \circ P \vert^W) \vert_Z^Z = \psi(P \circ A \circ P) \vert_Z^Z \notag \\
	&= \left( \psi(P) \vert_Z \circ R \circ \psi(P) \vert^Z \right) \Big\vert_Z^Z = R.
	\end{align}
	This proves that $\theta$ is a surjective algebra homomorphism. Since $Z$ is non-zero, from the SHAI property of $W$ it follows that $\theta$ is injective.
	
	Now let $A \in \mathscr{B}(X)$ be such that $A \in \Ker(\psi)$. Then $\psi(A)=0$ implies 
	\begin{align}
	\theta (P \vert^W \circ A \circ P \vert_W) &=  \psi (P \circ A \circ P) \vert_Z^Z = \left( \psi(P) \circ \psi(A) \circ  \psi(P) \right) \big\vert_Z^Z = 0.
	\end{align}
	Since $\theta$ is injective it follows that $P \vert^W A P \vert_W =0$ or equivalently $PAP =0$. We apply this in the following specific situation: We choose $x \in W = \Ran(P) \subseteq X$ and $\xi \in  X^*$ norm one vectors with $\langle x, \xi \rangle =1$. As $\psi$ is not injective, in particular we have $x \otimes \xi \in \mathscr{F}(X) \subseteq \Ker(\psi)$, consequently $P(x \otimes \xi)P =0$. Thus $0 = (P(x \otimes \xi)P)x = \langle Px, \xi \rangle Px = \langle x, \xi \rangle x =x$, a contradiction.
	
	Consequently $P \in \Ker(\psi)$ must hold, as required.
\end{proof}

We obtain the following corollary for Banach spaces of continuous functions, which can be viewed as a strengtening of the first part of \cite[Proposition~44]{lkad}.

\begin{Cor}
	Let $K$ be a compact Hausdorff space. Let $Y$ be a non-zero Banach space and let $\psi\colon  \mathscr{B}(C(K)) \rightarrow \mathscr{B}(Y)$ be a surjective, non-injective algebra homomorphism. Then $\overline{\mathscr{G}}_{c_0}(C(K)) \subseteq \Ker(\psi)$.
\end{Cor}

\begin{proof}
	If $C(K)$ has a complemented subspace isomorphic to $c_0$, then \cite[Proposition~1.2]{horvath2} and Proposition \ref{complshai} yield the claim. Now assume that $C(K)$ does not contain a complemented copy of $c_0$. By \cite[Corollary~2]{cembranos} this is equivalent to saying that $C(K)$ is a Grothendieck space. By \cite{diestel73} we thus have $\mathscr{X}(C(K)) \subseteq \mathscr{W}(C(K))$. By Pełczy\'{n}ski's theorem \cite[Theorem~1]{pel}, we also know that  $\mathscr{W}(C(K)) = \mathscr{S}(C(K))$. Consequently, with Lemma \ref{noninjbigkernel} we conclude
	\begin{align}
	\overline{\mathscr{G}}_{c_0}(C(K)) \subseteq \mathscr{X}(C(K)) \subseteq \mathscr{W}(C(K)) = \mathscr{S}(C(K)) \subseteq \mathscr{E}(C(K)) \subseteq \Ker(\psi),
	\end{align}
	which finishes the proof.
\end{proof}
	
We are now ready to prove Theorem A.

\begin{proof}[Proof of Theorem A]
	We prove by transfinite induction. Let $\lambda$ be a fixed infinite cardinal and let $E_{\lambda}$ be $c_0(\lambda)$, $\ell_{\infty}^c(\lambda)$ or $\ell_p(\lambda)$, where $p \in [1, \infty)$. Suppose $E_{\kappa}$ has the SHAI property for each cardinal $\kappa < \lambda$. \smallskip 
	
	Assume towards a contradiction that there is a non-zero Banach space $Y$ and a surjective, non-injective algebra homomorphism $\psi\colon  \mathscr{B}(E_{\lambda}) \rightarrow \mathscr{B}(Y)$. We first observe that $\Ker(\psi) \neq \mathscr{B}(E_{\lambda})$, since $Y$ is non-zero. Secondly $Y$ cannot be finite-dimensional. Indeed, otherwise $\mathscr{B}(Y)$ would be finite-dimensional, hence $\Ker(\psi)$ were finite-codimensional in $\mathscr{B}(E_{\lambda})$. But $E_{\lambda} \cong E_{\lambda} \oplus E_{\lambda}$ therefore $\mathscr{B}(E_{\lambda})$ cannot have finite-codimensional proper two-sided ideals, as it follows, for example, from applying \cite[Propositions~1.9 and 2.3]{ringfinofopalgs} and \cite[Proposition~1.3.34]{Dales} successively. Fix a cardinal $\kappa < \lambda$. As $E_{\kappa}$ is isomorphic to a complemented subspace of $E_{\lambda}$, there is an~idempotent $P_{(\kappa)} \in \mathscr{B}(E_{\lambda})$ with $\Ran(P_{(\kappa)}) \cong E_{\kappa}$. Clearly $P_{(\kappa)} \notin \mathscr{S}_{E_{\kappa}}(E_{\lambda})$, hence by Theorem \ref{kappasingingen} it follows that $\mathscr{S}_{E_{\kappa^+}}(E_{\lambda}) \subseteq \overline{\mathscr{G}}_{E_{\kappa}}(E_{\lambda})$. As $E_{\kappa}$ has the SHAI property by the inductive hypothesis, we conclude from Proposition \ref{complshai} that
	\begin{align}\label{forinduction}
	\mathscr{S}_{E_{\kappa^+}}(E_{\lambda}) \subseteq \overline{\mathscr{G}}_{E_{\kappa}}(E_{\lambda}) \subseteq \Ker(\psi).
	\end{align}

	We \textit{claim} that $\mathscr{S}_{E_{\lambda}}(E_{\lambda}) \subseteq \Ker(\psi)$. We consider three cases:
	
	\begin{enumerate}
		\item $\lambda = \omega$;
		\item $\lambda$ is a successor cardinal;
		\item $\lambda$ is uncountable and not a successor cardinal.
	\end{enumerate}
	
	(1) If $\lambda = \omega$ then $E_{\lambda} = c_0$ or $E_{\lambda} = \ell_p$, where $p \in [1, \infty]$. As Lemma \ref{noninjbigkernel} yields that we have $\mathscr{E}(E_{\lambda}) \subseteq \Ker(\psi)$, the claim follows from Corollary \ref{classicalideal}.
	
	(2) If $\lambda$ is a successor cardinal then $\lambda = \kappa^+$ for some cardinal $\kappa < \lambda$. From \eqref{forinduction} we thus conclude
	\[ \mathscr{S}_{E_{\lambda}}(E_{\lambda}) = \mathscr{S}_{E_{\kappa^+}}(E_{\lambda}) \subseteq \Ker(\psi). \]

	(3) Lastly, let $\lambda$ be an uncountable cardinal which is not a successor of any cardinal. By \eqref{forinduction} we clearly have $\mathscr{S}_{E_{\kappa}}(E_{\lambda}) \subseteq \mathscr{S}_{E_{\kappa^+}}(E_{\lambda}) \subseteq \Ker(\psi)$ for each $\kappa < \lambda$. As $\Ker(\psi)$ is (norm-)closed, in view of \cite[Lemma~3.15]{jksch} we obtain
	\[ \mathscr{S}_{E_{\lambda}}(E_{\lambda}) = \overline{ \bigcup\limits_{\kappa < \lambda} \mathscr{S}_{E_{\kappa}}(E_{\lambda})} \subseteq \Ker(\psi). \]
	Hence the claim is proved. Observe that $\mathscr{S}_{E_{\lambda}}(E_{\lambda})$ is the unique maximal ideal of $\mathscr{B}(E_{\lambda})$ by Corollary \ref{complhommaxideal} and \cite[Theorem~3.14]{jksch}, or \cite[Theorem~1.1]{jksch} in the case of $E_{\lambda} = \ell_{\infty}^c(\lambda)$. Since $\Ker(\psi)$ is a proper, two-sided ideal of $\mathscr{B}(E_{\lambda})$, we must have $\mathscr{S}_{E_{\lambda}}(E_{\lambda}) = \Ker(\psi)$. This is however equivalent to $\mathscr{B}(E_{\lambda}) / \mathscr{S}_{E_{\lambda}}(E_{\lambda}) \cong \mathscr{B}(Y)$, which is impossible. Indeed; the left-hand side is simple, since $\mathscr{S}_{E_{\lambda}}(E_{\lambda})$ is a maximal two-sided ideal of $\mathscr{B}(E_{\lambda})$; whereas $\mathscr{B}(Y)$ is not simple since $Y$ is infinite-dimensional. Thus $\psi$ must be injective and the proof is complete.
\end{proof}

\subsection{The SHAI property is not a three-space property}

We remind the reader that it follows from \cite[Proposition~1.6]{horvath2} that if $E$ is a Banach space and $F$ is a complemented subspace of $E$ such that both $F$ and $E/F$ have the SHAI property then $E$ itself has the SHAI property. Until now, however, we were not able to determine whether this holds without insisting on $F$ being complemented in $E$.\smallskip

In light of a recent deep result due to Koszmider and Laustsen (\cite{lkad}) and with the aid of Theorem A, we can conclude now that this is not the case.\smallskip

 Briefly speaking, an Isbell--Mr\'owka space $K_{\mathscr{A}}$ was constructed in \cite{lkad} such that the algebra of operators of $C_0(K_{\mathscr{A}})$ ---the Banach space of continuous functions on $K_{\mathscr{A}}$ vanishing at infinity--- admits a character (see \cite[Theorem~2~(iii)]{lkad}). Let us recall some terminology and the details of the constriction, for more details we refer the reader to \cite[Section~1]{lkad}. Given an almost disjoint family $\mathscr{A} \subseteq [\mathbb{N}]^{\omega}$, consider the Banach space $$\mathscr{X}_{\mathscr{A}}:= \overline{\spanning} \{ \mathds{1}_{B}\colon B \in \mathscr{A} \cup [\mathbb{N}]^{< \omega} \}.$$ Clearly $c_0 \subseteq \mathscr{X}_{\mathscr{A}}.$ In fact, $\mathscr{X}_{\mathscr{A}}$ is a closed, self-adjoint, non-unital subalgebra of the $C^*$-algebra $\ell_{\infty}$.\smallskip

On the one hand, a routine argument shows that $\mathscr{X}_{\mathscr{A}} / c_0$ and $c_0(\mathscr{A})$ are isometrically isomorphic as (non-unital) $C^*$-algebras. On the other hand, the Gel'fand--Naimark Theorem yields a (non-compact) locally compact, Hausdorff, scattered space $K_{\mathscr{A}}$ such that $\mathscr{X}_{\mathscr{A}}$ and $C_0(K_{\mathscr{A}})$ are isometrically isomorphic as $C^*$-algebras, hence as Banach spaces. Topological spaces of the form $K_{\mathscr{A}}$ are called \emph{Isbell--Mr\'owka spaces}.\smallskip

Armed with Theorem A and \cite[Theorem~2~(iii)]{lkad}, we are ready to demonstrate that the SHAI property fails to be a three-space property in every possible way.\pagebreak

\begin{Prop}{\,}\label{not3sp}
	\begin{romanenumerate}
		\item \label{mr1} There is a Banach space $E$ with the SHAI property that has a closed subspace $F$ which does not have the SHAI property.
		\item \label{mr2} There is a Banach space $E$ with the SHAI property that has a closed subspace $F$ such that $E/F$ does not have the SHAI property.
		\item \label{mr3} There is a Banach space $E$ with a subspace $F$ such that both $F$ and $E/F$ have the SHAI property, but $E$ does not.
	\end{romanenumerate}
\end{Prop}

\begin{proof}
	\eqref{mr1} Let $E:= \ell_{\infty}$ and let $F$ be an isomorphic copy of the James space $J_2$ in $E$. (Such an $F$ we can always find due to separability of $J_2$.) Now $E$ has the SHAI property by \cite[Proposition~1.2]{horvath2}, but $F$ does not have the SHAI property by \cite[Example~3.4~(2)]{horvath2}.\smallskip
	
	\eqref{mr2} Let $E:= \ell_1$ and let $F$ be a closed subspace of $E$ such that $E/F \cong J_2$. (Such $F$ exists again by separability of $J_2$.) Now $E$ has the SHAI property by \cite[Proposition~1.2]{horvath2}, but $E/F$ does not, as seen above.\smallskip
	
	\eqref{mr3} By \cite[Theorem~2]{lkad}, there is an uncountable almost disjoint family $\mathscr{A} \subseteq [\mathbb{N}]^{\omega}$ such that $\mathscr{B}(C_0(K_{\mathscr{A}}))$ has a character, where $K_{\mathscr{A}}$ is the Isbell--Mr\'owka space corresponding to $\mathscr{A}$. Consequently by \cite[Lemma~2.2]{horvath2} the Banach space $C_0(K_{\mathscr{A}})$ does not have the SHAI property. On the one hand, as $\mathscr{X}_{\mathscr{A}} \cong C_0(K_{\mathscr{A}})$, it follows that $\mathscr{X}_{\mathscr{A}}$ does not have the SHAI property either. On the other hand $\mathscr{X}_{\mathscr{A}} / c_0 \cong c_0(\mathscr{A})$, and it follows from \cite[Proposition~1.2]{horvath2} and Theorem A that both $c_0$ and $c_0(\mathscr{A})$ have the SHAI property. Setting $E:= \mathscr{X}_{\mathscr{A}}$ and $F:= c_0$ concludes the proof.
\end{proof}

\subsection{Open problems}

We conclude this section with some open problems.\smallskip

As discussed before, $C(K)$-spaces may or may not have the SHAI property. Indeed, the spaces $c_0(\lambda)$, $C(\beta \mathbb{N})\cong \ell_\infty$ have the SHAI property (Theorem A), whereas $C[0, \omega_1]$, $C_0(K_{\mathscr{A}})$ and $C(K_0)$ do not have the SHAI property; here $K_0$ is a Koszmider space without isolated points (\cite[Example 2.4 (3) and (6)]{horvath2}). Further naturally arising problems are:

\begin{Que}Does the space $C(K)$ have the SHAI property, where
\begin{romanenumerate}
    \item $K = [0,1]$,
    \item $K = [0, \omega^{\omega}]$,
    \item \label{stcremainder} $K = \beta \mathbb N\setminus \mathbb N$,
    \item $K = \beta \Gamma$ for an uncountable discrete space $\Gamma$?
\end{romanenumerate}
\end{Que}
Let us also ask a question that, if answered negatively, would make various arguments concerning SHAI easier.

\begin{Que}
Do there exist Banach spaces $X$ and $Y$ with $X$ separable and $Y$ non-separable such that there exists a surjective (but not injective) algebra homomorphism $\psi\colon  \mathscr{B}(X) \rightarrow \mathscr{B}(Y)$?
\end{Que}
We have been told by W.~B.~Johnson that this very question had been considered before by various researchers.\smallskip

\begin{Fun}
This work was supported by the Czech Science Foundation (GA\v{C}R) [grant number 19-07129Y; RVO 67985840].
\end{Fun}

\begin{Ack}
The authors are grateful to Tommaso Russo (Prague) for many enlightening conversations about the contents of the paper, especially for the discussions regarding Lemma \ref{seprangeomega1}. They also thank the anonymous referee for carefully reading the paper and spotting the small gap in the proof of Lemma \ref{notsigmasotclosed}. The referee's comments and suggestions helped to improve the quality of the paper a great deal.
\end{Ack}


\bibliographystyle{plain}
\bibliography{kitoltendo}

\begin{bibdiv}
\begin{biblist}

\bib{acgjm}{article}{
      author={Argyros, S.~A.},
      author={Castillo, J.~F.},
      author={Granero, A.~S.},
      author={Jimen\'{e}z, M.},
      author={Moreno, J.~P.},
       title={Complementation and embeddings of $c_0({I})$ in {B}anach spaces},
        date={2002},
     journal={Proc. London Math. Soc.},
      volume={85},
      number={3},
       pages={742\ndash 768},
}

\bib{cembranos}{article}{
      author={Cembranos, P.},
       title={{$C(K,E)$} contains a complemented copy of $c_0$},
        date={1984},
     journal={Proc. Amer. Math. Soc.},
      volume={91},
      number={4},
       pages={556\ndash 558},
}

\bib{Dales}{book}{
      author={Dales, H.~G.},
       title={Banach algebras and automatic continuity},
   publisher={Oxford University Press Inc., New York},
        date={2000},
}

\bib{dawsideal}{article}{
      author={Daws, M.},
       title={Closed ideals in the {B}anach algebra of operators on classical
  non-separable spaces},
        date={2006},
     journal={Math. Proc. Camb. Phil. Soc.},
      volume={140},
      number={317},
       pages={317\ndash 332},
}

\bib{diestel73}{book}{
      author={Diestel, J.},
       title={Grothendieck spaces and vector measures},
   publisher={Vector and operator valued measures and applications (Proc.
  Sympos., Alta, Utah, 1972), 97--108. Academic Press, New York},
        date={1973},
}

\bib{dudley}{article}{
      author={Dudley, R.~M},
       title={On sequential convergence},
        date={1964},
     journal={Trans. Amer. Math. Soc.},
      volume={112},
       pages={483\ndash 507},
}

\bib{edmit}{article}{
      author={\`Edel'\v{s}te\u{\i}n, I.~S.},
      author={Mitjagin, B.~S.},
       title={The homotopy type of linear groups of two classes of {B}anach
  spaces. ({R}ussian)},
        date={1970},
     journal={Funkcional. Anal. i Priložen},
      volume={4},
      number={3},
       pages={61\ndash 72},
}

\bib{fschzs}{article}{
      author={Freeman, D.},
      author={Schlumprecht, Th.},
      author={Zs\'{a}k, A.},
       title={Closed ideals of operators between classical sequence spaces},
        date={2017},
     journal={Bull. Lond. Math. Soc.},
      volume={49},
      number={5},
       pages={859\ndash 876},
}

\bib{fschzs2}{article}{
      author={Freeman, D.},
      author={Schlumprecht, Th.},
      author={Zs\'{a}k, A.},
       title={The number of closed ideals of $\mathcal{L}(\ell_p \oplus
  \ell_q)$},
        date={2020},
     journal={\texttt{arXiv:2006.15415}},
}

\bib{hkr2020}{article}{
      author={H{\'a}jek, P.},
      author={Kania, T.},
      author={Russo, T.},
       title={Separated sets and {A}uerbach systems in {B}anach spaces},
        date={2020},
     journal={Trans. Amer. Math. Soc.},
      volume={373},
       pages={6961\ndash 6998},
}

\bib{hajek}{book}{
      author={H\'{a}jek, P.},
      author={Santaluc\'{i}a, V.~Montesinos},
      author={Vanderwerff, J.},
      author={Zizler, V.},
       title={Biorthogonal systems in banach spaces, cms books in mathematics},
   publisher={Springer Science+Business Media, LLC},
        date={2008},
}

\bib{horvath2}{article}{
      author={Horv\'{a}th, B.},
       title={When are full representations of algebras of operators on
  {B}anach spaces automatically faithful?},
        date={2020},
     journal={Studia Math.},
      volume={253},
      number={3},
       pages={259\ndash 282},
}

\bib{jameson}{book}{
      author={Jameson, G. J.~O.},
       title={Topology and normed spaces},
   publisher={Chapman and Hall Ltd., London},
        date={1974},
}

\bib{jksch}{article}{
      author={Johnson, W.~B.},
      author={Kania, T.},
      author={Schechtman, G.},
       title={Closed ideals of operators on and complemented subspaces of
  {B}anach spaces of functions with countable support},
        date={2016},
     journal={Proc. Amer. Math. Soc.},
      volume={144},
      number={10},
       pages={4471\ndash 4485},
}

\bib{lkad}{article}{
      author={Koszmider, P.},
      author={Laustsen, N.~J.},
       title={A {B}anach space induced by an almost disjoint family, admitting
  only few operators and decompositions},
        date={2020},
     journal={\texttt{arXiv:2003.03832}},
}

\bib{koethe}{article}{
      author={K{\"o}the, G.},
       title={Hebbare lokalkonvexe {R}{\"a}ume},
        date={1966},
     journal={Math. Annalen},
      volume={165},
       pages={181\ndash 195},
}

\bib{laustsenmax1}{article}{
      author={Laustsen, N.~J.},
       title={Maximal ideals in the algebra of operators on certain {B}anach
  spaces},
        date={2002},
     journal={Proc. Edinb. Math. Soc.},
      volume={45},
      number={3},
       pages={523\ndash 546},
}

\bib{ringfinofopalgs}{article}{
      author={Laustsen, N.~J.},
       title={On ring-theoretic (in)finiteness of {B}anach algebras of
  operators on {B}anach spaces},
        date={2003},
     journal={Glasgow Math. J.},
      volume={45},
      number={1},
       pages={11\ndash 19},
}

\bib{laustsenloy}{article}{
      author={Laustsen, N.~J.},
      author={Loy, R.~J.},
       title={Closed ideals in the {B}anach algebra of operators on a {B}anach
  space},
        date={2005},
     journal={In Topological algebras, their applications, and related topics,
  Banach Center Publ.},
      volume={67},
       pages={245\ndash 264},
}

\bib{Megginson}{book}{
      author={Megginson, R.~E.},
       title={An introduction to banach space theory},
   publisher={Springer-Verlag, New York},
        date={1998},
}

\bib{ordman}{article}{
      author={Ordman, E.~T.},
       title={Convergence almost everywhere is not topological},
        date={1966},
     journal={Amer. Math. Monthly},
      volume={73},
      number={2},
       pages={182\ndash 183},
}

\bib{pel}{article}{
      author={Pełczy\'{n}ski, A.},
       title={On strictly singular and striclty cosingular operators. {I}.
  {S}trictly singilar and striclty cosingular operators in ${C(S)}$-spaces},
        date={1965},
     journal={Bull. Acad. Polon. Sci. Sér. Sci. Math. Astronom. Phys.},
      volume={13},
       pages={31\ndash 36},
}

\bib{pelczynski1}{article}{
      author={Pełczyński, A.},
       title={Projections in certain {B}anach spaces},
        date={1960},
     journal={Studia Math.},
      volume={19},
       pages={206\ndash 228},
}

\bib{pietsch}{book}{
      author={Pietsch, A.},
       title={Operator ideals},
   publisher={North-Holland Publishing Company},
        date={1979},
}

\bib{rodsal}{article}{
      author={Rodriguez-Salinas, B.},
       title={On the complemented subspaces of $c_0({I})$ and $\ell_p({I})$ for
  $1 < p < \infty$},
        date={1994},
     journal={Atti Sem. Mat. Fis. Univ. Modena},
      volume={42},
       pages={399\ndash 402},
}

\bib{rosenthalmeasure}{article}{
      author={Rosenthal, H.},
       title={On relatively disjoint families of measures, with some
  applications to {B}anach space theory},
        date={1970},
     journal={Studia Math.},
      volume={37},
       pages={13\ndash 36},
}

\bib{Ryan}{book}{
      author={Ryan, R.~A.},
       title={Introduction to tensor products of banach spaces},
   publisher={Springer-Verlag, London},
        date={2002},
}

\bib{schzs}{article}{
      author={Schlumprecht, Th.},
      author={Zs\'{a}k, A.},
       title={The algebra of bounded linear operators on $\ell_p \oplus \ell_q$
  has infinitely many closed ideals},
        date={2018},
     journal={J. Reine Angew. Math.},
      volume={735},
       pages={225\ndash 247},
}

\bib{whitley}{article}{
      author={Whitley, R.~J.},
       title={Strictly singular operators and their conjugates},
        date={1964},
     journal={Trans. Amer. Math. Soc.},
      volume={113},
       pages={252\ndash 261},
}

\end{biblist}
\end{bibdiv}

\end{document}